\documentclass[12pt,a4paper,reqno]{amsart}
\usepackage{lscape}
\usepackage[latin1]{inputenc}
\usepackage{tikz}

\usetikzlibrary{shapes,arrows}

\tikzstyle{decision} = [diamond, draw, fill=white,
    text width=4.5em, text badly centered, node distance=3cm, inner sep=0pt]
\tikzstyle{block} = [rectangle, draw, fill=white,
    text width=12em, text centered, rounded corners, minimum height=4em]
\tikzstyle{line} = [draw, -latex']
\tikzstyle{cloud} = [draw, ellipse,fill=white, node distance=3cm,
    minimum height=2em]

\newtheorem{thm}{Theorem}[section]
 \newtheorem{cor}[thm]{Corollary}
 \newtheorem{lem}[thm]{Lemma}
 \newtheorem{prop}[thm]{Proposition}
 \theoremstyle{definition}
 \newtheorem{defn}[thm]{Definition}
 
 \theoremstyle{remark}
 
 \theoremstyle{notation}
 
 \numberwithin{equation}{section}

\newcommand{\eps}{\epsilon}

\newcommand{\beq}{\begin{equation}}
\newcommand{\eeq}{\end{equation}}
\newcommand{\imn}{\mathbf{i}}
\newcommand{\gdim}{\mathbf{n}}

  \newcommand{\s}{{S}}
 \newcommand{\A}{\mathcal{A}}
 
\newcommand{\f}{\mathcal{F}}
\newcommand{\I}{\mathcal{I}}

  \newcommand{\h}{\mathfrak{h}}
 \newcommand{\bneg}{\mathfrak{b}}

\newcommand{\nneg}{\mathfrak{n}}

 \newcommand{\g}{\mathfrak{g}}

 \newcommand{\Ad}{\textrm{Ad}}
 \newcommand{\ad}{\textrm{ad}}
  
 \newcommand{\Z}{\mathbb{Z}}

 \newcommand{\Lie}{\mathfrak{L}}
 \newcommand{\q}{z}
 
\newcommand{\gauge}{\mathcal{N}}

 \newcommand{\lop}[1]{\mathfrak{L}(#1)}
\newcommand{\bil}[2]{{\langle #1 | #2\rangle}}

\newcommand{\npos}{\mathfrak{n}^+}
\newcommand{\wdt}[1]{\widetilde{#1}}
\newcommand{\wht}[1]{\widehat{#1}}

\begin{document}

\title[Frobenius manifolds and  $W$-algebras]
 {Frobenius manifolds from subregular classical $W$-algebras }

\author[Yassir Dinar]{Yassir Ibrahim Dinar }

\address{Faculty of Mathematical Sciences, University of Khartoum,  Sudan. Email: dinar@ictp.it.}
\email{dinar@ictp.it}


\begin{abstract}
We obtain  algebraic  Frobenius manifolds from classical $W$-algebras associated to subregular  nilpotent elements in simple Lie algebras of type $D_r$ where $r$ is even and $E_r$. The resulting  Frobenius manifolds  are  certain  hypersurfaces in the total spaces of semiuniversal deformation of simple hypersurface singularities  of the same types.
\end{abstract}
\maketitle
\tableofcontents

\section{Introduction}

A \textbf{Frobenius manifold} $M$ is a manifold with the structure of Frobenius algebra on the tangent space $T_zM$ at any point $z\in M$ with a  flat invariant bilinear form $(.,.)$ and an identity $e$ plus some compatibility conditions \cite{DuRev}. We say $M$ is semisimple or massive if $T_zM$ is semisimple for generic $z$.  Locally, in the flat coordinates $(t^1,...,t^r)$,  the structure of Frobenius manifold  is encoded  in a potential $F(t^1,...,t^r)$ satisfying  a system of partial differential equations  known in topological field theory as the  Witten-Dijkgraaf-Verlinde-Verlinde (WDVV) equations. We consider Frobenius  manifolds  where the quasihomogeneity condition takes the from
\begin{equation}
\sum_{i=1}^r d_i t_i \partial_{t^i} \mathbb{F}(t) = \left(3-d \right) \mathbb{F}(t)
\end{equation}
where  \textbf{$e=\partial_{t^{r-1}}$ and $d_{r-1}=1$}. This condition defines {\bf the degrees} $d_i$ and {\bf the charge} $d$ of   $M$. If $\mathbb{F}(t)$ is an algebraic function we call $M$ an
\textbf{algebraic Frobenius manifold}. \textbf{Dubrovin conjecture} on classification of algebraic Frobenius manifolds is stated as follows:
semisimple  irreducible algebraic Frobenius manifolds
with positive degrees $d_i$ correspond to quasi-Coxeter (primitive)  conjugacy
classes in irreducible Coxeter groups. A quasi-Coxeter conjugacy class in an irreducible Coxeter group is  a Conjugacy class  which has no representative in a proper Coxeter subgroup \cite{CarClassif}.

There are two major  results  support  the conjecture. First, the  conjecture arises from studying the algebraic solutions to associated equations of isomonodromic deformation of algebraic Frobenius manifolds \cite{DuRev},\cite{DMD}. It  leads to quasi-Coxeter conjugacy classes in Coxeter groups by considering the  classification of  finite orbits of the braid group action on tuple of reflections obtained in \cite{STF}.  Therefore, it remains the problem of constructing  all these algebraic  Frobenius manifolds. Second, Dubrovin constructed polynomial Frobenius structures on the orbit spaces of  Coxeter groups \cite{DCG} using the work of \cite{saito}. Then  Hertling  \cite{HER} proved that  these are all possible \textbf{polynomial Frobenius manifolds}. The isomonodromic deformation of Polynomial Frobenius manifolds lead to   Coxeter conjugacy classes  \cite{DuRev}.

The classification of polynomial Frobenius manifolds reveals a relation between the order and eigenvalues of the  conjugacy class, and the charge and degrees of the corresponding Frobenius manifold. More precisely, If the order of the conjugacy class is $\kappa+1$ and the eigenvalues are $\exp {2\eta_i \pi \imn\over \kappa+1}$ then the charge of the Frobenius manifold is $\kappa-1\over \kappa+1$ and the degrees are $\eta_i+1\over \kappa+1$.  We depend on this \textbf{weak relation} in considering  a new examples of algebraic Frobenius manifolds.

 In \cite{mypaper} we continue the work of  \cite{PAV} and we began to develop a construction of algebraic Frobenius manifolds using  Drinfeld-Sokolov reduction. This means we restrict ourself to conjugacy classes in Weyl groups. The examples  obtained  correspond, in the notations of \cite{CarClassif},  to the conjugacy classes $D_4(a_1)$ and $F_4(a_2)$. In \cite{mypaper1} we  succeeded to uniform the  construction of  all polynomial  Frobenius manifolds.   In this work we uniform the   construction of  algebraic Frobenius manifolds which  correspond to  the conjugacy classes $D_r(a_1)$ where $r$ is even and $E_r(a_1)$.

 In order to formulate the main results of this work, let us recall the relation between subregular nilpotent elements and deformation of simple hypersurface singularities \cite{sldwy2}. Let $\g$ be a simple Lie algebra of  type  $D_r$ where $r$ is even or $E_r$. Fix a \textbf{subregular nilpotent element} $e$  in  $\g$.  By definition a nilpotent element is called subregular if its centralizer in $\g$ is of  dimension $r+2$. We fix, by using the Jacobson-Morozov theorem,   a semisimple element $h$ and a nilpotent element $f$ such that $\A=\{e,h,f\}$ is an $sl_2$-triple, i.e
 \begin{equation}
 [h,e]=2 e;~~~[h,f]=-2f;~~~[e,f]=h.
 \end{equation}
  The action of $\A$ decompose $\g$ to $r+2$ irreducible $\A$-submodules
\begin{equation}
\g=\bigoplus_{i=1}^{r+2} V^i
\end{equation}
Let $\dim V^i=2\eta_i+1$. We call the set
\begin{equation}
Wt(e)=\{\eta_i: i=1,...,r+2\}
\end{equation}
\textbf{the weights of $e$}. Let $\eta_0+1$ be the Coxeter number of $\g$. The set $Wt(e)$, under our choice of a total  order, together with $\eta_0$ are given in the following table.
\begin{center}
\begin{tabular}{|c|c|c|c|c|}\hline
$\g$ & & \multicolumn{3}{|c|}{$Wt(e)$} \\
\hline & $\eta_0$& $\eta_1,\ldots,\eta_{r-1}$& $\eta_r$& $\eta_{r+1},\eta_{r+2}$ \\
\hline
$D_r$ & 2r-3& $1,3,\ldots,r-1;r-1,r+1,\ldots, 2r-5$ & $\textbf{r-3}$& 1 ,r-2\\
$E_6$&11 & $1,4,5,7,8$ & $\textbf{2}$&$3,5$\\
$E_7$ &17& $1,5,7,9,11,13$ & $\textbf{3}$& $5,8$\\
$E_8$ &29& $1,7,11,13,17,19,23$ & $\textbf{5}$& $9,14$\\
\hline  & \multicolumn{2}{|c|}{$Et(\g)$}& & \\\hline
\end{tabular}
\end{center}
We observe that the set $Et(\g)$ of the exponents of $\g$ is
\begin{equation}
Et(\g)=\{\eta_i: i=0,...,r-1\}.
\end{equation}

 Let  $G$ be the adjoint group of $\g$. By  Chevalley theorem, the algebra $S(\g^*)^{G}$ of invariant polynomials under the adjoint  action of $G$ is generated by  $r$ homogenous polynomials of degrees $\eta_i+1,~i=0,...,r-1$. We fix a homogenous generator $\chi^0,...,\chi^{r-1}$   of the algebra $S(\g^*)^{G}$ with degree $\chi^i$ equals $\eta_i+1$. Then let us consider the \textbf{adjoint quotient} map   \beq \chi=(\chi^0,\ldots,\chi^{r-1}):\g\to \mathbb{C}^r\eeq
 and define the \textbf{Slodowy slice}  \beq Q=e+\g^f;~~~\g^f:=\ker \ad~f.\eeq
 Then Brieskorn proved  that the restriction of  $\chi$ to $Q$ is semiuniversal  deformation of the simple hypersurface singularity $\mathcal{N}\cap Q$ \cite{bri} which is of the same type as $\g$.

Let us introduce the  following coordinates on  $Q$
\[\sum_{i=1}^{r+2} z^i X_{-\eta_i}^i+e\in {Q}\]
where $X_{-\eta_i}^i$ is a normalized  minimal weight vector of $V^i$.  We assign the degree $ 2 \eta_i+2$ to $z^i$. Then a simple modification of the work in \cite{sldwy1} we get the following
\begin{prop}
The map $\chi|_{Q}$ has rank $r-1$ at $e$. We can normalize  the  modules $V^i$ and choose  a homogenous generators  $\chi^0,...,\chi^{r-1}$   for  $S(\g^*)^{G}$  such that the restriction $t^i$ of $\chi^i$ to $Q$  with $i>0$ take the form
 \beq\label{inv gene}
t^i=z^i+ ~ {\textrm{ non linear terms}}.
 \eeq
In particular, setting \beq
t^{r+i}=z^{r+i},~i=0,1,2\eeq  we get a quasihomogenous coordinates $ (t^1,...,t^{r+2})$
 on $Q$ with degree $t^i$ equals degree $z^i$.
\end{prop}
We will call the coordinates $(t^1,\ldots, t^{r+2})$ on $Q$ obtained in this proposition  \textbf{Slodowy coordinates}. In this coordinates  the restriction of the quotient map to $Q$ take the form
\beq
\chi|_{Q}:(t^1,....,t^{r+2})\mapsto(t^0,t^1,...,t^{r-1})
\eeq
where $t^0$ is the restriction to ${Q}$ of the invariant polynomial $\chi^0$.

Let us consider the Lie-Poisson bracket $P$ on $\g$. The Dirac reduction of $P$ to $Q$ give a nontrivial Poisson bracket $\{.,.\}^{Q}$. It is know in the literature as the adjoint transverse Poisson structure to the nilpotent orbit of $e$. In \cite{DamSab} they prove the following
\begin{thm}
 The  matrix $F^{ij}(t)=\{t^i,t^j\}^{Q}$  is constant multiple of the matrix
\beq\begin{pmatrix}
  0 & 0 &  \\
  0 &  \Omega \\
\end{pmatrix}
\eeq
where $\Omega$ is a $3 \times 3$ matrix of the form
\beq\begin{pmatrix}

  0& {\partial t^0 \over \partial t^{r+2}} & -{\partial t^0\over \partial t^{r+1}} \\
 -{\partial  t^0\over \partial t^{r+2}} & 0 & {\partial t^0\over \partial t^{r}} \\
 {\partial  t^0\over \partial t^{r+1}} &- {\partial  t^0\over \partial t^{r}} & 0 \\
\end{pmatrix}.
\eeq
\end{thm}

Let $N\subset Q$ be the hypersurface  of dimension $r$ defined as follows 
\beq
N=\Big\{t\in Q:{ \partial t^0 \over \partial t^{r+2}} = {\partial t^0\over \partial t^{r+1}}=0\Big\}
\eeq

It will follow  that ${\partial t^0 \over \partial t^{r+2}}$ depends linearly on $t^{r+2}$ and  $ {\partial t^0\over \partial t^{r+1}}$ is a polynomial in $t^{r+1}$ of degree $r-2$ (resp. $2$ ) if  $\g$ is a Lie algebra of type $D_r$ (resp. $E_r$). In particular,  $(t^1,....,t^r)$ are well defined coordinates on $N$. Let $\kappa = \max Wt(e)$. Then we proved the following
\begin{thm}
  The space $N$ has a natural  structure of algebraic Frobenius manifold   with charge  $\kappa-1\over \kappa+1$ and degrees $\eta_i+1\over \kappa+1$,~ $i=1,...,r$.
\end{thm}

By natural we mean that it can be formulated entirely in terms of the representation theory of $\A$. The potential $\mathbb{F}$ of this Frobenius structure depends on the solution of the equation ${\partial t^0\over \partial t^{r+1}}=0$. The set
\beq
Et(e)=\{\eta_i: i=1,...,r\}
\eeq
plays the same role as the set $Et(\g)$ for polynomial Frobenius manifolds and we call it \textbf{the exponents of $e$}.
Here is some brief details about what we did in order to prove the theorem above.

\begin{enumerate}
\item We review  the relation between the nilpotent element and the conjugacy class. Using the  work of \cite{springer} and \cite{DelFeher}, we fix  a number $\rho$ such that $y_1=e+\rho X_{-\eta_\kappa}^\kappa$ is regular semisimple. Then $\h'=\ker y_1$ is Cartan subalgebra and it is known as \textbf{the opposite  Cartan subalgebra}. The element $w:=\exp {2\pi \imn\over \kappa+1}\ad~h$, acts on $\h'$ as a representative of the conjugacy class $D_r(a_1)$ (resp. $E_r(a_1)$) if  $\g$ is of type $D_r$ (resp. $E_r$).
\item We consider the standard local Poisson bracket on the loop algebra $\lop \g$. Then we use the subalgebra $\A$ to perform  Drinfeld-Sokolov reduction (\cite{mypaper},\cite{BalFeh1},\cite{gDSh2}) or equivalently the Dirac  reduction \cite{BalFeh1},\cite{mypaper2} to get a local Poisson structure $\{.,.\}^{\widetilde{Q}}$ on  Slodowy slice
   \[ {\widetilde{Q}}=e+ \lop {\g^f};.\]
   This Poisson bracket  is known in the literature as the \textbf{classical $W$-algebras} associated to the nilpotent orbit of $e$ \cite{BalFeh1}. The leading term of this Poisson bracket  is the adjoint transverse Poisson structure $\{.,.\}^{Q}$.

\item we preform  the Dirac reduction on  $\{.,.\}^{\widetilde{Q}}$ in order to obtain a local Poisson  structure  which admits a dispersionless limit. This is possible  on a loop space ${\widetilde{N}}:=\lop {N}$ of  the  hypersurface $N$ defined above. The new local Poisson bracket $\{.,.\}^{\widetilde{N}}$ is also a classical $W$-algebra. We call it \textbf{subregular classical $W$-algebra}.  Then from $\{.,.\}^{\widetilde{N}}$  we get a local  Poisson bracket of hydrodynamic type $\{.,.\}^{[0]}$.  In the coordinates  $(t^1,...,t^r)$ of $N$ we  have by definition
    \begin{equation}\{t^i(x),t^j(y)\}^{[0]} = g^{ij}(t(x))
  \delta'(x-y) + \Gamma^{ij}_{k}(t(x)) t^k_x \delta(x-y).\end{equation}
where $g^{ij}(t(x))$ and $\Gamma^{ij}_{k}(t(x))$ are polynomials in $t^i(x)$.

\item We transfer to finite dimensional geometry. We  check that  the matrix $g^{ij}(t(x))$ is nondegenerate. It follows that the nondegeneracy condition could be traced back to the existence of opposite Cartan subalgebra $\h'$.  Then from  Dubrovin-Novikov theorem, $g^{ij}(t)$ define a flat contravariant metric  on  $N$. Moreover, form the structure  of the matrix $g^{ij}(t)$,  it follows that the matrix $\partial_{t^{r-1}} g^{ij}$ define another flat  contravariant metric on $N$. Furthermore, the two metrics form   a flat pencil of metrics on $N$.

\item We prove that the flat pencil of metrics on $N$ satisfies the  quasihomogeneity and the regularity conditions. The proof depends on  the definition  of classical $W$-algebra and the structure of the set $Et(e)$.  Then we obtain  a Frobenius structure on $N$  using  a theorem and construction  due to  Dubrovin \cite{DFP}. The resulting  Frobenius manifold satisfy the weak relation.

\end{enumerate}

\section{Preliminaries}

\subsection{Flat pencil of metrics and Frobenius manifolds}
In this section we review the relation between the geometry  of flat pencil of metrics and the theory of Frobenius manifolds outlined  in \cite{DFP}.

Let $M$ be a smooth manifold of dimension $r$. A symmetric bilinear form $(. ,. )$ on $T^*M$ is called \textbf{contravariant
metric} if it is invertible on an open dense subset $M_0 \subset M$. In a local coordinates $(u^1, . . . , u^r)$, if we set
\beq
 g^{ij}(u)=(du^i, du^j);~ i, j = 1, . . . , r,
\eeq
Then  the inverse matrix $ g_{ij}(u)$ of $ g^{ij}(u)$ determines a metric $<. ,. >$ on $TM_0$. We define the \textbf{contravariant Levi-Civita connection} $\Gamma^{ij}_k$ for $(. ,. )$ by
\beq
\Gamma^{ij}_k:=-g^{is} \Gamma_{sk}^j
\eeq
where $\Gamma_{sk}^j$ is the  Levi-Civita connection of $<. ,. >$. We say the metric $(.,.)$ is flat if  $<. ,. >$ is flat.

Let  $g_1^{ij}(u)$ and $g_2^{ij}(u)$ be two contravariant flat metrics on $M$ and denote the corresponding Levi-Civita connections by $\Gamma_{1;k}^{ij}(u)$ and $\Gamma_{2;k}^{ij}(u)$, respectively.   We say $g_1^{ij}(u)$ and $g_2^{ij}(u)$ form a  \textbf{flat pencil of metrics} if
\begin{enumerate}
\item  $g_\lambda^{ij}(u):=g_2^{ij}(u)+\lambda g_1^{ij}(u)$ defines a flat metric on $T^*M$ for a generic $\lambda$ and,
\item  The  Levi-Civita connection of $g_\lambda^{ij}(u)$ is given by \[\Gamma_{\lambda ;k}^{ij}(u)=\Gamma_{2;k}^{ij}(u)+\lambda \Gamma_{1;k}^{ij}(u).\]\end{enumerate}

The flat pencil of metrics in this work is obtained by using the following lemma

\begin{lem} \cite{DCG}\label{flat pencil}
If for a contravariant flat metric $g^{ij}_2$ in some coordinate $(u^1,...,u^r)$ the entries of $g^{ij}_2(u)$ and   its Levi-Civita connection $\Gamma^{ij}_{2;k}$ depend linearly on $u^{r}$ then the metric
\beq
g^{ij}_1=\partial_{u^r} g^{ij}_2
\eeq
 with $g^{ij}_2$ form a flat pencil of metrics. The Levi-Civita connection of the metric $g^{ij}_1$ has the form
 \beq
 \Gamma^{ij}_{1;k}=\partial_{u^r} \Gamma^{ij}_{2;k}.
 \eeq
 \end{lem}

We are concern with the following  particular class of flat pencil of metrics.
\begin{defn} \label{def reg} A contravariant flat pencil of metrics on a manifold $M$ defined by  the matrices $g_1^{ij}$ and $g_2^{ij}$ is
called \textbf{ quasihomogenous of  degree} $d$ if there exists a
function $\tau$ on $M$ such that the vector fields
\begin{eqnarray}  E&:=& \nabla_2 \tau, ~~E^i
=g_2^{is}\partial_s\tau
\\\nonumber  e&:=&\nabla_1 \tau, ~~e^i
= g_1^{is}\partial_s\tau  \end{eqnarray} satisfy the following
properties
\begin{enumerate}
\item $ [e,E]=e$.

\item $ \Lie_E (~,~)_2 =(d-1) (~,~)_2 $.
\item $ \Lie_e (~,~)_2 =
(~,~)_1 $.
\item
$ \Lie_e(~,~)_1
=0$.
\end{enumerate}
Here, for example $\Lie_E$ denote the Lie derivative  along the vector field $E$ and $(~,~)_1$ denote the metric defined by the matrix  $g^{ij}_1$. In addition, the  quasihomogenous flat pencil of metrics is called \textbf{regular} if  the
(1,1)-tensor
\begin{equation}\label{regcond}
  R_i^j = {d-1\over 2}\delta_i^j + {\nabla_1}_i
E^j
\end{equation}
is  nondegenerate on $M$.
\end{defn}

\subsubsection{\textbf{Frobenius manifolds}}
A Frobenius algebra is a commutative associative algebra with unity $e$ and an invariant nondegenerate bilinear form $(.,.)$.
A \textbf{Frobenius manifold} is a manifold $M$ with
a smooth structure of Frobenius algebra on the tangent space $T_tM$ at
any point $t \in M $ with certain compatibility conditions \cite{DuRev}. Globally, we  require the metric $(.,.)$ to be flat and the unity vector field $e$ is constant with respect to it. In the flat  coordinates $(t^1,...,t^r)$ where $e={\partial\over \partial t^{r-1}}$ the compatibility conditions implies  that there exist a function $\mathbb{F}(t^1,...,t^r)$ such that
\[ \eta_{ij}=(\partial_{t^i},\partial_{t^j})=  \partial_{t^{r-1}}
\partial_{t^i}
\partial_{t^j} \mathbb{F}(t)\]
and the structure constants of the Frobenius algebra is given by
\[ C_{ij}^k=\eta^{kp}  \partial_{t^p}\partial_{t^i}\partial_{t^j} \mathbb{F}(t)\]
where $\eta^{ij}$ denote the inverse of the matrix $\eta_{ij}$.
In this work, we consider Frobenius manifolds where the quasihomogeneity condition takes the form
\begin{equation}
\sum_{i=1}^r d_i t^i \partial_{t^i} \mathbb{F}(t) = \left(3-d \right) \mathbb{F}(t);~~~~d_{r-1}=1.
\end{equation}
 This condition defines {\bf the degrees} $d_i$ and {\bf the charge} $d$ of  the Frobenius structure.
 If $\mathbb{F}(t)$ is an algebraic function we call $M$ an \textbf{algebraic Frobenius manifold}.
The  associativity of Frobenius
algebra implies the potential $\mathbb{F}(t)$ satisfy a system of partial differential equations which appears in topological field theory
and  called WDVV equations:
\begin{equation} \label{frob}
 \partial_{t^i}
\partial_{t^j}
\partial_{t^k} \mathbb{F}(t)~ \eta^{kp} ~\partial_{t^p}
\partial_{t^q}
\partial_{t^n} \mathbb{F}(t) = \partial_{t^n}
\partial_{t^j}
\partial_{t^k} \mathbb{F}(t) ~\eta^{kp}~\partial_{t^p}
\partial_{t^q}
\partial_{t^i} \mathbb{F}(t).
  \end{equation}

The following theorem gives a connection between the geometry of Frobenius manifolds and flat pencil of metrics.
\begin{thm}\cite{DFP}\label{dub flat pencil}
A contravariant quasihomogenous regular  flat pencil of metrics of degree $d$ on a manifold $M$ defines a Frobenius structure on $M$ of  degree $d$.
 \end{thm}

It is well known that from a Frobenius manifold we always have a flat pencil of metrics but it does not necessary satisfy the regularity condition \eqref{regcond} \cite{DFP}. Locally, in the coordinates defining equation  \eqref{frob}, the flat pencil of metrics is
found from the equations \begin{eqnarray}\label{frob eqs} \eta^{ij}&=&g_1^{ij} \\
\nonumber g_2^{ij}&=&(d-1+d_i+d_j)\eta^{i\alpha}\eta^{j\beta}
\partial_\alpha
\partial_\beta \mathbb{F}
\end{eqnarray}
This flat pencil of metric is quasihomogenous  of degree $d$ with $\tau =t^1$. Furthermore, we have
\begin{equation}
E=\sum_i d_i t^i {\partial_{t^i}};~~~e={\partial_{t^{r-1}}}.
\end{equation}

\subsection{Local Poisson brackets}

In this section  we fix notations and we review  the Dirac reduction for local Poisson brackets on loop spaces.

Let $M$ be a manifold. The loop spaces $\lop M$ of $M$ is the space of smooth functions from the circle to $M$. A local Poisson bracket $\{.,.\}$ is  a Poisson bracket on the space of local  functional on $\lop M$. If we choose a local coordinates  $(u^1,...,u^r)$ then $\{.,.\}$  is a finite summation   of the form
\begin{eqnarray} \label{genLocPoissBra}\label{genLocBraGen}\{u^i(x),u^j(y)\}&=&
\sum_{k=-1}^\infty \epsilon^k \{u^i(x),u^j(y)\}^{[k]}\\\nonumber
\{u^i(x),u^j(y)\}^{[k]}&=&\sum_{s=0}^{k+1} A_{k,s}^{i,j}(u(x))
\delta^{(k-s+1)}(x-y),
 \end{eqnarray}
where $\epsilon$ is just a parameter, $A_{k,s}^{i,j}(u(x))$ are homogenous polynomials in $\partial_x^j
u^i(x)$ of degree $s$ when we  assign
$\partial_x^j u^i(x)$ degree $j$, and $\delta(x-y)$ is the Dirac  delta function defined by
\[\int_{S^1} f(y) \delta(x-y) dy=f(x).\]
In particular, the first terms can be written
as follows
\begin{eqnarray}\label{loc poiss}
  \{u^i(x),u^j(y)\}^{[-1]} &=& F^{ij}(u(x))\delta(x-y) \\\nonumber
  \{u^i(x),u^j(y)\}^{[0]} &=& g^{ij}(u(x)) \delta' (x-y)+ \Gamma_k^{ij}(u(x)) u_x^k \delta (x-y).
\end{eqnarray}
   where $g^{ij}(u)$, $F^{ij}(u)$ and $\Gamma_k^{ij}(u)$ are smooth functions on the finite dimensional space $M$.  It follows from the definition that the matrix $F^{ij}(u)$  defines a Poisson structure on
$M$. We say the Poisson bracket admits a \textbf{dispersionless limit} if $F^{ij}(u)=0$ and   $\{u^i(x),u^j(y)\}^{[0]}\neq 0$. In this case  $\{u^i(x),u^j(y)\}^{[0]}$ defines a
 Poisson bracket on $\lop M$ known as \textbf{Poisson bracket of
hydrodynamic type}. We call it \textbf{nondegenerate } if $\det (g^{ij}(u))\neq 0$ in an open dense subset of $M$.

\begin{thm}\cite{DN}\label{DN thm}
In the notations given above. If $\{u^i(x),u^j(y)\}^{[0]}$ defines  a nondegenerate Poisson brackets of
hydrodynamic type then the matrix $g^{ij}(u)$ defines a contravariant flat metric on $M$ and
$\Gamma_{k}^{ij}(u)$ is its  contravariant Levi-Civita connection.
\end{thm}

\subsubsection{\textbf{Dirac reduction}}
Assume we have a local Poisson bracket on the loop space $\lop M$ of a manifold $M$. Let $N\subset M$ be a submanifold of dimension $m$. Then under some assumptions the Poisson bracket can be reduced to $N$ using Dirac reduction. For this end
 we assume $N$ is defined by the equations $u^\alpha=0$ for $\alpha=m+1,...,r$. We introduce three types of indexes; capital letters $I,J,K,...=1,..,r$,
small letters $i,j,k,...=1,....,m$ which parameterize the
submanifold $N$ and Greek letters
$\alpha,\beta,\delta,...=m+1,...,r$.
\begin{prop}\cite{mypaper}\label{dirac fromula}In the notations of equations \eqref{loc poiss}.  Assume  the minor matrix
$F^{\alpha \beta}$ is nondegenerate. Then Dirac reduction is well defined on $\lop {N}$ and gives a local Poisson bracket.
If we write the leading terms of the reduced Poisson bracket in the form
\begin{eqnarray}
  \{u^i(x),u^j(y)\}^{[-1]}_N &=& \widetilde{F}^{ij}(u)\delta(x-y), \\
  \{u^i(x),u^j(y)\}^{[0]}_N &=& \widetilde{g}^{ij}_0 (u)\delta' (x-y)+ \widetilde{\Gamma}_k^{ij} u_x^k \delta (x-y).
\end{eqnarray}
Then
\begin{equation}
\widetilde{F}^{ij}=(F^{ij}-F^{i\beta} F_{\beta\alpha} F^{\alpha j},
)\end{equation}
\begin{equation}
\widetilde{g}^{ij}= g^{ij}_0-g^{i\beta} F_{\beta\alpha}F^{\alpha
j}+F^{i \beta}F_{\beta \alpha} g^{\alpha \varphi} F_{\varphi
\gamma} F^{\gamma j}-F^{i\beta} F_{\beta \alpha} g^{\alpha j},
\end{equation}
and
\begin{equation}
\begin{split}
\widetilde{\Gamma}^{ij}_k u_x^k&= \big(\Gamma^{ij}_k - \Gamma^{i\beta}_k  F_{\beta \alpha} F^{\alpha j} + F^{i \lambda} F_{\lambda \alpha} \Gamma^{\alpha \beta}_k F_{\beta \varphi} F^{\varphi j}-F^{i\beta} F_{\beta \alpha} \Gamma^{\alpha j}_k \big) u_x^k\\
&-\big(g^{i\beta} - F^{i\lambda} F_{\lambda \alpha} g^{\alpha \beta} \big)\partial_x(F_{\beta \varphi} F^{\varphi j})
\end{split}
\end{equation}
and the other terms could be found by solving certain recursive equations.
\end{prop}

\begin{cor}\label{cor dirac}
If the entries  $F^{i \alpha}=0$ on $N$, then the reduced Poisson bracket on $\lop N$ will have the same leading terms, i.e
 \begin{eqnarray} \widetilde{F}^{ij}&=& F^{ij}.\\\nonumber  \widetilde{g}^{ij}&=&g^{ij}.\\\nonumber
 \widetilde{\Gamma}^{ij}_k &=& \Gamma^{ij}_k. \end{eqnarray}
\end{cor}

\section{Subregular nilpotent elements }
We  review some facts about the theory of  subregular nilpotent elements in  simple Lie algebras and  the related structure of opposite Cartan subalgebra.

 Let $\g$ be simple Lie algebra of rank $r$. We  assume   the Lie algebra $\g$ is of type $D_r$ where $r$ is even or $E_r$. This assumption is due to the fact that for simply laced Lie algebra, the  opposite Cartan subalgebra for a subregular nilpotent element exists only for these types of Lie algebra.

Let us fix  a subregular  nilpotent element $e\in \g$. By definition, a nilpotent element is called subregular  if $\g^e:=\ker \ad \, e$ has dimension equal to $r+2$ \cite{COLMC}.  Using  the Jacobson-Morozov theorem, we fix a semisimple element $h$ and a nilpotent element $f$ in $\g$ such that $\{e,h,f\}$ generate a $sl_2$-subalgebra $\A\subset \g$ satisfying
\begin{equation}
[h,e]=2 e, ~~~ [h,f]=-2f,~~~[e,f]=h.
\end{equation}
Let us consider  the adjoint representation  of $\A$ on $\g$. Then $\g$ decomposes  to irreducible $\A$-submodules
\begin{equation}\label{decompo}
\g=\bigoplus_{i=1}^{r+2} V^i.
\end{equation}
Let $\dim V^i=2\eta_i+1$ and assume $V^1$ is isomorphic  to $\A$ as a vector space. We call the set
\begin{equation}
Wt(e)=\{\eta_i: i=1,...,r+2\}
\end{equation}
\textbf{the weights of $e$}. Let $\eta_0+1$ be the Coxeter number of $\g$. The set $Wt(e)$, under our choice of a total  order, together with $\eta_0$ are given in the following table.
\begin{center}\label{wtstable}
\begin{tabular}{|c|c|c|c|c|}\hline
$\g$ & & \multicolumn{3}{|c|}{$Wt(e)$} \\
\hline & $\eta_0$& $\eta_1,\ldots,\eta_{r-1}$& $\eta_r$& $\eta_{r+1},\eta_{r+2}$ \\
\hline
$D_r$ & 2r-3& $1,3,\ldots,r-1;r-1,r+1,\ldots, 2r-5$ & $\textbf{r-3}$& 1 ,r-2\\
$E_6$&11 & $1,4,5,7,8$ & $\textbf{2}$&$3,5$\\
$E_7$ &17& $1,5,7,9,11,13$ & $\textbf{3}$& $5,8$\\
$E_8$ &29& $1,7,11,13,17,19,23$ & $\textbf{5}$& $9,14$\\
\hline  & \multicolumn{2}{|c|}{$Et(\g)$}& & \\\hline
\end{tabular}
\end{center}
We observe that the set $Et(\g)$ of exponents of $\g$ is given by
\begin{equation}
Et(\g)=\{\eta_i: i=0,...,r-1\}.
\end{equation}

We emphasis that  the statements and proofs in this work  depend  explicitly on the total ordering of the set $Wt(e)$ and $Et(\g)$ given in this table.

 We fix  on $\g$ the invariant bilinear form $\bil . .$ such that $\bil e f=1$. We  normalize the  decomposition \eqref{decompo} and we fix a basis for each $V^i$ by using the following proposition.

 \begin{prop}\label{reg:sl2:normalbasis}
 There exists a decomposition of $\g$ into a sum of irreducible $\A$-submodules \[\g=\oplus_{i=1}^{r+2} V^i\] in such a way that there is a basis $X_I^i, I=-\eta_i,-\eta_i+1,...,\eta_i$ for each $V^i, ~i=1,\ldots,r+2$ satisfying the following relations
\begin{equation}\label{sl2expand}
X_{I}^i={1\over (\eta_i+I)!} \ad \,e^{\eta_i+I}~X_{-\eta_i}^i~ ,~~~~I=-\eta_i,-\eta_i+1,\ldots, \eta_i.
\end{equation}
and
\begin{eqnarray}\label{sl2relation}
\ad \, h\,X_I^i&=& 2I X_I^i.\\\nonumber
\ad \, e\,X_I^i&=& (\eta_i+I+1) X_{I+1}^i.\\\nonumber
\ad \, f\, X_I^i &=& (\eta_i -I+1) X_{I-1}^i.
\end{eqnarray}
Furthermore
\beq\label{sl2bilinear}
<X_I^i,X_J^j>=\delta_{i,j}\delta_{I,-J} (-1)^{\eta_i-I+1}{2\eta_i\choose \eta_i-I}.
\eeq
\end{prop}
\begin{proof}
  The proof  is similar to the proof of proposition 2.3 in  \cite{mypaper1}, since for all simple Lie algebras, except  $D_4$, there are at most two weights of the same value.  The Lie algebra $D_4$ has three weights equal one. We give  such a normalization for  $D_4$  in section \ref{D4}.
\end{proof}
We observe that the normalized basis for $V^1$ are \[X_1^1=-e,~X_0^1=h,~X_{-1}^1=f.\]

We recall that the semisimple element  $h$ define the following  $\mathbb{Z}$-grading on $\g$ and it is called \textbf{the Dynkin grading}
\begin{equation}
\g=\bigoplus_{i\in \Z} \g_i; ~~\g_i=\{q\in \g: ad~h(q)=i q\}.
\end{equation}
We observe that  $\g_i=0$ if $i$ is odd.
 \subsection{Opposite Cartan subalgebra}
 The following theorem summarize the relation between the   subregular nilpotent element $e$ and a quasi-Coxeter conjugacy class in the Weyl group of $\g$. To simplify the notations let  \textbf{$\kappa$ denote the maximum weight $\eta_{r-1}$}.
\begin{thm}\label{opposite}
There exists a nonzero element  $X'\in \g_{-2\kappa}$ such that the element \[y_1=e+X'\] is regular semisimple. Let $\h'$ be the  Cartan subalgebra containing $y_1$, i.e \[\h'=\ker \ad~y_1\] and  consider the adjoint group element $w$ defined by  \[w=\exp {2\pi \imn\over \kappa+1}\ad~h.\] Then $w$ acts on $\h'$ as a representative of a regular quasi-Coxeter conjugacy class  in the Weyl group acting on $\h'$. The conjugacy class is of type $D_r(a_1)$ (resp. $E_r(a_1)$) if $\g$ is of type $D_r$ (resp. $E_r$). Furthermore, the element $y_1$ can be completed to a basis  $y_i,~i=1,\ldots,r$ for $\h'$
having the form
\[ y_i=v_i+u_i, ~~u_i\in \g_{2\eta_i},~ v_i\in \g_{2\eta_i-2(\kappa+1)}\]
and such that $y_i$ is an eigenvector of $w$ with eigenvalue $\exp {2\pi \imn\eta_i\over \kappa+1}$.
\end{thm}
\begin{proof} The  proof for a Lie algebra of type $E_i,~i=6,7,8$ is obtained by Springer \cite{springer}.  The case of a Lie algebra of type $D_r$, the proof is given  in the appendix of \cite{DelFeher}.
\end{proof}

  We fix an element  $X'$ satisfy the hypothesis of  the theorem above. In the literature, the element  $y_1=e+X'$  is called a \textbf{cyclic element} and   $\h'=\ker \ad ~y_1$ is called the \textbf{opposite Cartan subalgebra}. We will call the set
  \beq
 Et(e):=\{\eta_i, ~i=1,...r\}\subset  Wt(e)
  \eeq    \textbf{the exponents of $e$} as it plays the same role of the exponents of  $(\g)$ to the regular nilpotent elements \cite{mypaper1}. Let us
  consider a basis $y_i=u_i+v_i$ for $\h'$ satisfying the hypothesis of the theorem above. We normalize them by using the following theorem
\begin{prop}\label{norm min}
The basis $y_i=u_i+v_i$ can  be chosen in such  a way that
\beq
u_i=-X_{\eta_i}^i, ~i=1,\ldots,r
\eeq
and
\beq
v_1=X'=\rho X^{r-1}_{-\kappa}
\eeq
for some nonzero number $\rho$
\end{prop}
\begin{proof}
From the construction we know that $u_1=e=-X_{1}^1$. Then, it is easy to see that $u_i, ~i=1,...,r$ generate a commutative subalgebra of $\g^e$. But $X^{i}_{\eta_i}$ are homogenous basis  for $\g^e$. Hence for  a Lie algebra of type  $E_8$ and $D_r,~r>4$ the normalization of $u_i, i=2,...,r$ and $v_1$ follows  from the structure of the set $Wt(e)$.  For a Lie algebra of type $E_6$, $E_7$ and $D_4$ we obtained  such  normalization by direct computations.
\end{proof}

Let us consider the matrix $A_{i,j}$ of the invariant bilinear form on $\h'$ under the basis $y_i=-X^i_{\eta_i}+v_i$.
\beq\label{Opp Cartan}
A_{ij}=\bil {y_i}{y_j}=-\bil {X_{\eta_i}^i}{v_j}-\bil {v_i}{X_{\eta_j}^i};~~i,j=1,\ldots,r.
\eeq
We know from the theory of Cartan subalgebras that the matrix $A_{ij}$ is nondegenerate. Some useful properties  we gain from $\h'$ are summarized in the following proposition.

\begin{prop}\label{opp Cartan prop}
The matrix $A_{ij}$ is antidiagonal  with respect to the set $Et(e)$ in the sense that  \[A_{ij}=0, ~{\rm if}~\eta_i+\eta_j\neq\kappa+1.\]
Therefore, after totally reordering the set $Et(e)$ in the form
\[\mu_1\leq \mu_2\ldots \leq \mu_r,\] we have the property \beq
\mu_j+\mu_{r-j+1}=\kappa+1=\eta_1+\eta_{r-1} ~,j=1,...,r.\eeq
\end{prop}
\begin{proof}
We will use the fact that the matrix $\bil . .$ is a nondegenerate invariant bilinear form on $\h'$. Hence for any element $y_i$ there exists an element  $y_j$ such that $\bil {y_i}{y_j}\neq 0$. But if  $w$ is the  quasi-Coxeter element we defined in theorem  \ref{opposite} then the  equality
\[\bil {y_i}{y_j}=\bil {w y_i}{w y_j}=\exp {2(\eta_i+\eta_j)\pi \imn \over \kappa+1}\bil {y_i}{y_j} \]
implies that in case $\bil {y_i}{y_j}\neq 0$ we must have $\eta_i+\eta_j=\kappa+1$. Hence, the  matrix $A_{ij}$ is  antidiagonal  with respect to the set $Et(e)$.
\end{proof}

In the remainder of this paper let  \textbf{$a$ denote the element  $X_{-\kappa}^{r-1}$.}

\begin{prop}\label{Gold}
The commutators of $a$ and  $X_{\eta_i}^i$ satisfy the relation
\begin{equation}
{\bil {[a,X_{\eta_i}^i]}{X_{\eta_j-1}^j}\over 2\eta_j }+{\bil {[a,X_{\eta_j}^j]}{X_{\eta_i-1}^i}\over 2 \eta_i}=  {1\over \rho}A_{ij}
\end{equation}
for all $i,j=1,\ldots,r$. Here the nonzero number $\rho$ is the same as  in proposition \ref{norm min}.
\end{prop}

\begin{proof}

We note that  the commutator  of $y_1= e+\rho X_{-\kappa}^{r-1}$ and $y_i=v_i-X_{\eta_i}^i$ gives the relation
\begin{equation}
[e,v_i]=\rho [a,X_{\eta_i}^i], ~ i=1,...,r.
\end{equation}
This  in turn give the following equality for every  $i,~j=1,...,r$
\begin{eqnarray}
\bil {\rho[a,X_{\eta_i}^i]}{X_{\eta_j-1}^j}&=&\bil {[e,v_i]}{X_{\eta_j-1}^j}=-\bil {v_i}{[e,X_{\eta_j-1}^j]}\\\nonumber &=& -2 \eta_j \bil {v_i}{X_{\eta_j}^j}
\end{eqnarray}
but then
\begin{equation}
{\bil {[a,X_{\eta_i}^i]}{X_{\eta_j-1}^j}\over 2\eta_j }+{\bil {[a,X_{\eta_j}^j]}{X_{\eta_i-1}^i}\over 2 \eta_i}= -{1\over \rho}(\bil {v_i}{X_{\eta_j}^j}+\bil {v_j}{X_{\eta_i}^i})= {1\over \rho} A_{ij}.
\end{equation}
\end{proof}

\subsection{Slodowy coordinates}

We  review the relation between subregular nilpotent elements  and deformation of simple hypersurface singularities \cite{sldwy2}.

 Let $G$ be the adjoint group of $\g$. By  Chevalley theorem, the algebra $S(\g^*)^{G}$ of invariant polynomials under the adjoint  action of $G$ is generated by  $r$ homogenous polynomials of degrees $\eta_i+1,~i=0,...,r-1$. The inclusion homomorphism \[S(\g^*)^{G}\hookrightarrow S(g^*),\] is dual to a morphism \[\chi:\g\to \g/{G}\] called \textbf{the adjoint quotient}. We fix a homogenous generator $\chi^0,...,\chi^{r-1}$   of the algebra $S(\g^*)^{G}$ with degree $\chi^i$ equals $\eta_i+1$. Then the adjoint quotient  map is given by  \beq \chi=(\chi^0,\ldots,\chi^{r-1}):\g\to \mathbb{C}^r.\eeq
The fiber $\mathcal{N}:=\chi^{-1}(\chi(0))$ is called the \textbf{nilpotent variety of $\g$}, it consists of all  nilpotent elements of $\g$.

We define  the \textbf{Slodowy slice} to be  the affine subspace  \beq Q=e+\g^f\eeq
where $\g^f=\ker \ad~f$.
Then Brieskorn proved  that the restriction of  $\chi$ to $Q$ is semiuniversal  deformation of the simple hypersurface singularity $\mathcal{N}\cap Q$ which is of the same type as $\g$. Let us introduce the  following coordinates on  $Q$
\[\sum_{i=1}^{r+2} z^i X_{-\eta_i}^i+e\in {Q}\] and  assign degree $ 2 \eta_i+2$ to $z^i$.

\begin{prop}(\cite{sldwy1},section 7.2)
Let $\chi^i,~i=0,...,r-1$ be a homogenous generators of the ring $S(\g^*)^{G}$. Then the restriction of $\chi^i$ to $Q$  will be quasihomogenous of degree $2\eta_i+2$.
\end{prop}

\begin{prop}\label{sldwy coord}
The map $\chi|_{Q}$ has rank $r-1$ at $e$. We can normalize  the  modules $V^i$ and choose  a homogenous generators  $\chi^0,...,\chi^{r-1}$   for  $S(\g^*)^{G}$  such that the restriction $t^i$ of $\chi^i$ to $Q$  with $i>0$ take the form
 \beq\label{inv gene}
t^i=z^i+ ~ {\textrm{ non linear terms}}.
 \eeq
In particular, setting \beq
t^{r+i}=z^{r+i},~i=0,1,2\eeq  we get a quasihomogenous coordinates $ (t^1,...,t^{r+2})$
 on $Q$ with degree $t^i$ equals degree $z^i$.
\end{prop}
 \begin{proof}
Let  $\chi^0,...,\chi^{r-1}$ be a homogenous generators for $S(\g^*)^{G}$  and  denote $t^1,...,t^{r-1}$ their  restriction to $Q$. Using   the structure of the set $Wt(e)$ and $Et(\g)$ together with  the isolatedness of the singularity of $\mathcal{N}\cap Q$, Slodowy proved the following \cite{sldwy1} (section 8.3). We can  choose a coordinates $v^1,v^2,v^3$ from $z^1,....,z^{r+2}$ of degrees $2\eta_r+2,2\eta_{r+1}+2,2\eta_{r+2}+2$, respectively,  such that  \[(t^1,\ldots, t^{r-1},v^1,v^2,v^3)\] are homogenous coordinates on $Q$. Therefore, the statement follows upon proving that $v_i=z^{r+i}$ for $i=0,1,2$ when considering the normalization of proposition \ref{reg:sl2:normalbasis}. This is obvious for the Lie algebra $E_8$ since the numbers in the set $Wt(e)$ are all different. It is also true for the Lie algebra $D_r$, $r$ is even, since the  restriction of the first invariant $\chi^1$ to $S$ is, up to constant, equal to $z^1$. For the Lie algebra $E_6$ (respectively $E_7$) we verify by direct computation that the invariant $\chi^3$ (resp. $\chi^2$) depends explicitly on $z^3$ (resp. $z^2$).
 \end{proof}
We will call the coordinates $(z^1,\ldots, z^{r+2})$ on $Q$ obtained in this proposition  \textbf{Slodowy coordinates}. We observe  that in this coordinates the quotient map take the form
\beq
\chi|_{Q}:(t^1,....,t^{r+2})\mapsto(t^0,t^1,...,t^{r-1})
\eeq
where $t^0$ is the restriction to ${Q}$ of the invariant polynomial $\chi^0$. Setting $t^1,...,t^r$ equal zero in $t^0$ we get, from the quasihomogeneity,  a polynomial function  $f(t^r,t^{r+1},t^{r+2})$ of the form shown in the table below (we lower the index for convenience). Here $c_1,c_2,c_3,c_4$ are some constants.  Note that  the hypersurface $\mathcal{N}\cap Q$ will be given by setting $f(t^r,t^{r+1},t^{r+2})=0$. Moreover, from  the isolatedness of the singularity the numbers  $c_1,c_2,c_3$ are nonzero constants. The constant  $c_4$ could be eliminated by change of variables to obtain  the standard equation  defining the simple hypersurface singularity.

 \begin{center}\label{plystable}
\begin{tabular}{|c|c|}\hline
$\g$ &  $f(t_r,t_{r+1},t_{r+2})$\\

\hline
$D_r$ & $c_1 t_{r+1}^{r-1}+c_2 t_{r+1} t_{r}^2+c_3 t_{r+2}^2+ c_4 t_{r+1}^{r\over 2} t_{r}$\\
$E_6$& $ c_1 t_{6}^4+ c_2 t_{7}^3+ c_3 t_{8}^2+c_4 t_{6}^2 t_{8}$\\
$E_7$ & $c_1 t_{7}^3 t_{8}+c_2 t_{8}^3+c_3 t_{9}^2$\\
$E_8$ & $c_1 t_8^5+c_2 t_9^3+c_3 t_{10}^2$\\\hline

\end{tabular}
\end{center}

\section{Drinfeld-Sokolov reduction}

In this section we review the construction of the classical $W$-algebra associated to the nilpotent element $e$ using Drinfeld-Sokolov reduction.

 We use the Dynkin grading of $e$ to define the following subalgebras    \begin{eqnarray} \bneg&:=&\bigoplus_{i\leq 0} \g_i,\\\nonumber \nneg&:=&\bigoplus_{i\leq
{-2}}\g_i=[\bneg,\bneg].\end{eqnarray}
Then  we consider  the action of  the adjoint group $\gauge$ of $\lop {\nneg}$ on $\lop \g$ defined by
\begin{eqnarray}
q(x)&\rightarrow& \exp \ad~ s(x)( \partial_x+q(x))-\partial_x
\end{eqnarray}
where $s(x) \in \lop \nneg,~~q(x)\in \lop \g$.

Let us extend the invariant bilinear form from $\g$ to $\lop \g$ by setting
\begin{equation} (u|v)=\int_{S^1}\bil {u(x)}{v(x)} dx,~ u,v \in \lop M.
\end{equation}
Then we identify $\lop\g$ with $\lop \g^*$ by means of this bilinear form.  We define the gradient $\delta \f (q)$ for a functional $\f$ on $\lop\g$
 to be the unique element in
$\lop\g$ satisfying
\begin{equation}
\frac{d}{d\theta}\f(q+\theta
\dot{s})\mid_{\theta=0}=\int_{S^1}\langle\delta \f|\dot{s}\rangle dx
~~~\textrm{for all } \dot{s}\in \lop\g.
\end{equation}
We fix on $\lop \g$ the  following Poisson bracket
\begin{eqnarray}\label{bih:stru on g}
\{\f[q(x)],\I [q(y)]\}&=&{1\over \eps}\big(\delta\f(x)|[\eps \partial_x+q(x),\delta \I(x)]\big)
\end{eqnarray}
for every functional $\f$ and $\I$  on $\lop \g$.
\begin{prop}(\cite{mypaper})\label{DS as momentum }
The action of $\gauge$ on $\lop\g$ with Poisson bracket $\{.,.\}$ is
 Hamiltonian. It admits a momentum map $J$ to be
the projection \[J:\lop\g\to\lop \npos\] where  $\npos$ is the image of $\nneg$ under the Killing map. Moreover, $J$ is $\Ad^*$-equivariant.
\end{prop}

We take $e$ as a regular value of $J$. Since $\bneg$ is the orthogonal complement to $\nneg$ under $\bil . .$, we get the affine space  $S=J^{-1}(e)=\lop\bneg+e$. Moreover, it follows from the Dynkin grading that  the isotropy group  of $e$ is $\gauge$. Let $R$ be the ring of invariant differential polynomials of $S$ under the action of $\gauge$. Then,   from Marsden-Ratiu reduction theorem, the set $\mathcal{R} $ of functionals on $\s$ which have densities in the ring $R$ is closed under the Poisson brackets $\{.,.\}$.

Let us define the space ${\widetilde{Q}}$  to be the Slodowy slice
\beq
{\widetilde{Q}}:=e+\lop {\g^f}.
\eeq
The following proposition  identifies the space  $S/\gauge$  with  ${\widetilde{Q}}$.
\begin{prop}\cite{mypaper}
The space ${\widetilde{Q}}$
is a cross section for the action of $\gauge$ on $\s$, i.e for any element $q(x)+e\in \s$ there is a unique element $s(x) \in
\lop \nneg$ such that  \begin{equation}\label{gauge fix} z(x)+e=\exp \ad~s(x)( \partial_x+q(x))-\partial_x \in {\widetilde{Q}}.\end{equation} Therefore, the entries of $z(x)$ are generators of the ring
$R$.
\end{prop}
Hence, the space ${\widetilde{Q}}$ has a Poisson structure  $\{.,.\}^{\widetilde{Q}}$ from $\{.,.\}$. This Poisson  bracket is known in the literature as  \textbf{classical $W$-algebra} associated to $e$. For a formal definition of classical $W$-algebras see \cite{fehercomp}.

Let us  obtain the linear terms of the invariants $z^i(x)$. We introduce a parameter $\tau$ and write
\begin{eqnarray*}
q(x)+ e&=& \tau \sum_{i=1}^{r+2} \sum_{I=0}^{\eta_i} q_{i}^I X_{-I}^{i}+e\in \s\\
z(x)+e &=&\tau\sum_{i=1}^{r+2} z^i(x) X_{-\eta_i}^i+e\in {\widetilde{Q}}\\
s(x)&=&\tau\sum_{i=1}^{r+2}\sum_{I=1}^{\eta_i} s_{i}^{I}(x) X_{-I}^{i} \in \lop\nneg.
\end{eqnarray*}
Then equation \eqref{gauge fix} expands to
\begin{equation}\label{gauge fixing data }
\begin{split}
\sum_{i=1}^{r+2} z^i(x) X_{-\eta_i}^i+&\sum_{i=1}^{r+2}\sum_{I=1}^{\eta_i}(\eta_i-I+1) s_{i}^{I} X_{-I+1}^{i}=\\&\sum_{i=1}^{r+2} \sum_{I=0}^{\eta_i} q_{i}^I(x) X_{-I}^{i}-\sum_{i=1}^{r+2}\sum_{I=1}^{\eta_i} \partial_x s_{i}^{I}(x) X_{-I}^{i}+\mathcal{O}(\tau).
\end{split}
\end{equation}
Hence,  any invariant $z^i(x)$ will take the from
\begin{eqnarray}\label{leading terms1}
z^i(x)&=&q_i^{\eta_i}-\partial_x s_i^{\eta_i}+\mathcal{O}(\tau)\\\nonumber
&=& q_i^{\eta_i}(x)-\partial_x q_{i}^{\eta_i-1}+\mathcal{O}(\tau).
\end{eqnarray}
Furthermore, using     $\bil {e}{f}=1$ we get

\begin{eqnarray}
z^1(x)&=& q_1^1(x) -\partial_x s_1^1+ \tau \bil {e}{[s_i^1(x) X_{-1}^i, q_{i}^0 X_0^i]}\\\nonumber & &+{1\over2}\tau \bil{e}{[s_i^1(x) X_{-1}^i,[s_i^I(x) X_{-1}^i,e]]}\\\nonumber
       &=& q_1^1(x)- \partial_x q_1^0(x)+{1\over 2} \tau \bil {e}{[s_i^1(x) X_{-1}^i, q_{i}^0 X_0^i]}                              \\\nonumber
&=& q_1^1(x) -\partial_x q_1^0(x)+{1\over 2} \tau\sum_i (q_i^0(x))^2\bil{X_0^i}{ X_0^i}.
\end{eqnarray}
The invariant $z^1(x)$ is known in the literature as the  \textbf{Virasoro density}.

We observe that the reduced Poisson  structure could be obtained   as follows. We write the coordinates of ${\widetilde{Q}}$ as a differential polynomials in the coordinates of $S$ using   equation \eqref{gauge fix} and then we apply the Leibnitz rule. The Leibnitz rule for  $u,v \in R$  have the following form
\begin{equation}
\{u(x),v(y)\}={\partial u(x)\over \partial (q_i^I)^{(m)}}\partial_x^m\Big({\partial v(y)\over \partial (q_j^J)^{(n)}} \partial_y^n\big(\{q_i^I(x),q_j^J(y)\}\big)\Big).
\end{equation}

Our analysis for the Poisson brackets will relay on the quasihomogeneity of the invariants  $z^i(x)$ in  the coordinates of $q(x)\in \lop \bneg$ and their derivatives.
\begin{lem}\label{lin inv poly}
If we assign degree $2 J+2l+2$ to  $\partial_x^l(q_i^J(x))$  then $z^i(x)$ will be quasihomogenous of degree $2\eta_i+2$. Furthermore, each invariant $z^i(x)$ depends linearly only on $q_i^{\eta_i}(x)$ and $\partial_x q_{i}^{\eta_i-1}(x)$, i.e
\beq
z^i(x)= q_i^{\eta_i}(x)-\partial_x q_{i}^{\eta_i-1}+\textrm{nonlinear terms}.
\eeq
Furthermore
\beq
z^1(x)= q_1^1(x) -\partial_x q_1^0(x)+{1\over 2} \sum_i (q_i^0(x))^2\bil{X_0^i}{ X_0^i}.
\eeq
In particular, $z^i(x)$ with $i\neq r-1$ does not  depend on $q^{\kappa}_{r-1}(x)$ or its derivatives.
\end{lem}
 We fix the following notations for the leading terms of the Poisson bracket
\begin{eqnarray} \{z^i(x),z^j(y)\}^{\widetilde{Q}}&=&
\sum_{k=-1}^\infty \epsilon^k \{z^i(x),z^j(y)\}^{[k]}_1 \\\nonumber
 \end{eqnarray}
 where
\begin{eqnarray}\label{reducedPB notations}
  \{z^i(x),z^j(y)\}^{[-1]} &=& F^{ij}(z(x))\delta(x-y) \\\nonumber
  \{z^i(x),z^j(y)\}^{[0]} &=& g^{ij}(z(x)) \delta' (x-y)+ \Gamma_{k}^{ij}(z(x)) z_x^k \delta (x-y)\\\nonumber
   \end{eqnarray}

\subsection{The nondegeneracy  condition}

  We want  to prove  that the  minor matrix $g^{mn}(z),~m,n=1,\ldots,r$ is nondegenerate generically on ${\widetilde{Q}}$. For this end we define the matrix
 \beq
 g^{ij}_1(z)=\partial_{z^{r-1}} g^{ij}(z)
 \eeq
  and we will prove that the minor matrix $g^{mn}_1(z),~m,n=1,\ldots,r$   is lower antidiagonal  with respect to the set $Et(e)$, i.e
  \begin{equation*}
g^{mn}_1 = \left\{
\begin{array}{rl}
0 & \text{if } \eta_m+\eta_n<\kappa+1 \\
{1\over \rho}A_{mn} & \text{if } \eta_m+\eta_n=\kappa+1
\end{array} \right.
\end{equation*}
where the matrix ${1\over \rho} A_{mn}$ is defined  by equation  \eqref{Opp Cartan} and its properties were obtained in proposition \ref{opp Cartan prop}.

 We recall that $z^{r-1}$ is the coordinate of the lowest weight root vector $a=X_{\kappa}^{r-1}$.  We denote  $\Xi_I^i$ the value $\bil {X_I^i}{X_I^i}$ and set
  \[[a,X_I^i]=\sum_j \Delta_I^{ij} X_{I-\kappa}^j.\]
   The fact that  $z^{r-1}(x)$  is the only invariant which depends on $q_{r-1}^{\kappa}(x)$ implies that the invariant  $z^{r-1}(x)$ will appear in the expression of $  \{z^i(x),z^j(y)\}^{\widetilde{Q}}$ only if, when using the Leibnitz rule, we encounter  terms of the original Poisson bracket $\{.,.\}$ depend explicitly on  $q_{r-1}^{\kappa}(x)$. The later appear as a result  of the  following ``brackets''
\begin{equation}
[q_j^{\kappa-I}(x),q_i^I(y)]:=q_{r-1}^{\kappa}(x) {\Delta^{ij}_I\over \Xi_I^i} \delta(x-y).
\end{equation}
Hence the dependence of $  \{z^i(x),z^j(y)\}^{\widetilde{Q}}$ on $z^{r-1}(x)$ can be evaluated by imposing the Leibnitz rule on the ``brackets'' above. We get
{\small \begin{eqnarray*}
[z^m(x),z^n(y)]&=& \sum_{i,I;j}\sum_{l,h}{\Delta_I^{ij}\over \Xi_I^i}{\partial z^m(x) \over \partial (q_j^{\kappa-I})^{(l)} }\partial_x^l\Big({\partial z^n(y) \over \partial (q_i^{I})^{(h)} }\partial_y^h(q_{r-1}^{\kappa}(x)\delta(x-y))\Big)\\\nonumber
&=&  \sum_{i,I;j}\sum_{l,h,\alpha,\beta} (-1)^h{h\choose \alpha} {l\choose \beta}{\Delta_I^{ij}\over \Xi_I^i}{\partial z^m(x) \over \partial (q_j^{\kappa-I})^{(l)} }\Big( q_{r-1}^{\kappa}(x)\Big({\partial z^n(x) \over \partial (q_i^{I})^{(h)} }\Big)^{(\alpha)}\Big)^{(\beta)}\delta^{(h+l-\alpha-\beta)}(x-y).
\end{eqnarray*}}
Here we omitted the ranges of the indices since no confusion can arise. We observe that  the coefficient of $\delta'(x-y)$ of this expression which contributes to  the value of $g^{mn}(z)$ is given by
\begin{eqnarray}
\mathbb{B}(z^m,z^n)&=&\sum_{i,I,J}\sum_{h,l}(-1)^h (l+h) {\Delta_I^{ij}\over \Xi_I^i}q_{r-1}^{\kappa}(x){\partial z^m(x) \over \partial (q_j^{\kappa-I})^{(l)} }\Big({\partial z^n(x) \over \partial (q_i^{I})^{(h)} }\Big)^{h+l-1}
\end{eqnarray}
Obviously,  We get  $ g^{ij}_1(z)$ from  the expression
\begin{equation}\label{formula for dispesionless limit}\mathbb{A}(z^m,z^n) = \partial_{q_{r-1}^{\kappa}} \mathbb{B} =\sum_{i,I,J}\sum_{h,l}(-1)^h (l+h) {\Delta_I^{ij}\over \Xi_I^i}{\partial z^m(x) \over \partial (q_j^{\kappa-I})^{(l)} }\Big({\partial z^n(x) \over \partial (q_i^{I})^{(h)} }\Big)^{h+l-1}.
\end{equation}

 \begin{lem}  \label{homg of g1}  The matrix $\mathbb{A}(z^m,z^n)$ is lower antidiagonal with respect to $Wt(e)$ and the antidiagonal entries are constants. In other word  $\mathbb{A}(z^m,z^n)$ is constant  if $\eta_m+\eta_n\leq \kappa+1$
 and equals zero if $\eta_m+\eta_n<\kappa+1$.
 \end{lem}
\begin{proof}
We note that if  $z^m(x)$ and $z^n(x)$ are  quasihomogenous of  degree $2 \eta_m+2$ and $2\eta_n+2$, respectively, then $\mathbb{A}(z^m,z^n)$ will be quasihomogenous of degree \[ 2\eta_{m}+2+2\eta_n+2-(2\kappa+2)-4=2 \eta_m+2\eta_n-2\kappa-2.\]
The proof is complete.
\end{proof}

\begin{prop}\label{nondeg}
The minor matrix $g^{mn}_1,~m,n=1,\ldots,r$ is nondegenerate and its determinant is equal to the determinant of the matrix ${1\over \rho} A_{mn}$.
\end{prop}
\begin{proof}
We observe that, from our choice of  coordinates, the minor matrix $g^{mn}_1$ will be lower antidiagonal with respect to the set $Et(e)$. Hence, form the second part of proposition \ref{opp Cartan prop} we need only to prove that $\mathbb{A}(z^n,z^m)$ with $\eta_n+\eta_m=\kappa+1$ is nonzero constant. In this case $z^m$ and $z^n$ are quasihomogenous of degree $2\eta_m+2$ and $2\kappa-2\eta_m+4$, respectively.  The expression
 \begin{equation}\mathbb{A}(z^n,z^m) =\sum_{i,I,J}\sum_{h,l}(-1)^h (l+h) {\Delta_I^{ij}\over \Xi_I^i}{\partial z^n(x) \over \partial (q_j^{\kappa-I})^{(l)} }\Big({\partial z^m(x) \over \partial (q_i^{I})^{(h)} }\Big)^{h+l-1}
\end{equation}
gives the constrains
 \begin{eqnarray}
 2I+2\leq 2\eta_m+2 \\\nonumber
 2\kappa-2I+2 \leq 2\kappa-2 \eta_m+4
 \end{eqnarray}
which implies \[ \eta_m-1\leq I\leq \eta_m\]
Therefore the only possible values for the index $I$ in the expression of $\mathbb{A}(z^n,z^m)$ that make sense are $\eta_m$ and $\eta_m-1$. Consider  the partial summation of $ \mathbb{A}(z^n,z^m)$ when  $I=\eta_m$. The degree of $z^m(x)$ yields $h=0$ and that $z^m(x)$ depends linearly on $q_i^{\eta_m}(x)$.  But then equation \eqref{leading terms1} implies that $i$ is fixed and equals to $m$. A similar argument on $z^n(x)$ we find that the indices $l$ and $j$ are fixed and equal to $1$ and $n$, respectively. Hence,  the  partial summation when $I=\eta_m$ gives the value
 \[{\Delta_{\eta_m}^{mn}\over \Xi_{\eta_m}^m}{\partial z^n(x) \over \partial (q_n^{\kappa-\eta_m})^{(1)} }{\partial z^m(x) \over \partial (q_m^{\eta_m})^{(0)}} =-{\Delta_{\eta_m}^{mn}\over \Xi_{\eta_m}^m}.\]
We now turn to  the partial summation of $ \mathbb{A}(z^n,z^m)$ when  $I=\eta_m-1$.  The  possible values for $h$ are $1$ and $0$. When $h=0$  we get zero since $l$ and $h$ can only be zero. When $h=1$ we get, similar to the above calculation,  the value
 \[
(-1)  {\Delta_{\eta_m-1}^{mn}\over \Xi_I^i}{\partial z^n(x) \over \partial (q_n^{\kappa-\eta_m})^{(0)} }{\partial z^m(x) \over \partial (q_m^{\eta_m-1})^{(1)} }={\Delta_{\eta_m-1}^{mn}\over \Xi_{\eta_m-1}^m}.
\]  Hence we end with the expression \begin{eqnarray*}
\mathbb{A}(z^n,z^m)&=& {\Delta_{\eta_m-1}^{mn}\over \Xi_{\eta_m-1}^m}-{\Delta_{\eta_m}^{mn}\over \Xi_{\eta_m}^m}\\\nonumber
&=&{\bil {[a,X_{\eta_n}^n]}{X_{\eta_m-1}^m}\over 2\eta_m}+ {\bil {[a,X_{\eta_m}^m]}{X_{\eta_n-1}^n}\over 2\eta_n}={1\over \rho} A_{mn}
\end{eqnarray*}
where  the last equality was obtained in proposition \ref{Gold}. Hence, the  determinant of minor matrix $g^{mn}_1,~m,n=1,\ldots,r$ equals to the determinant of ${1\over \rho} A_{mn}$ which  is nondegenerate.
\end{proof}

\section{Dirac reduction}

In this section we extract information about $\{.,.\}^{\widetilde{Q}}$   using  the fact that the Poisson bracket $\{.,.\}^{\widetilde{Q}}$ can be obtained by preforming the Dirac reduction on  $\{.,.\}$.

Let $\gdim$ denote the dimension of $\g$. We use proposition \ref{reg:sl2:normalbasis} to fix a total order $\xi_I, ~I=1,\ldots,\gdim$ of the  basis $X^i_I$    such that:
\begin{enumerate}
\item The first  $ r +2$  are the lowest weight vectors in the following order    \beq X^1_{-\eta_1}<X^2_{-\eta_2}<\ldots<X^{r+2}_{-\eta_{r+2}}.\eeq
\item The matrix
\beq
\bil {\xi_I}{\xi_J},~ I,J=1,\ldots,\gdim
\eeq is antidiagonal.
\end{enumerate}
Let $\xi^I$ be  the dual basis of $\xi_I$ under $\bil . .$. We observe  that  if $\xi_I \in \g_{\mu}$ then $\xi^I \in \g_{-\mu}$. Let $c^{IJ}_K$ denotes  the structure constant of $\g$ under the dual basis and we set $\widetilde{g}^{IJ}=\bil{\xi^I}{\xi^J}$.  We extend the coordinates $z^i(x)$ on ${\widetilde{Q}}$ to all $\lop \g$ by setting for $q(x)\in \lop \g$
\begin{equation}
z^I(q(x))=\bil{q(x)-e}{\xi^I},~~ I=1,\ldots, \gdim.
\end{equation}
Then we  consider the  following matrix differential operator
\begin{equation}\label{bih operator}
\mathbb{F}^{IJ}=\epsilon \widetilde{g}^{IJ}\partial_x+ \widetilde{F}^{IJ},
\end{equation}
where, \[\widetilde F^{IJ}=\sum_{K}\big( c^{IJ}_K z^K(x)\big).\]
In this notations  the Poisson bracket   $\{.,.\}$  is given by
\begin{equation}
\{z^{I}(x),z^J(y)\}=\mathbb{F}^{IJ}{1\over \epsilon}\delta(x-y).
\end{equation}

For the rest of this section we   consider  three types of indices which have different ranges; capital letters $I,J,K,...=1,..,\gdim$,
small letters $i,j,k,...=1,....,r+2$  and Greek letters
$\alpha,\beta,\delta,...=r+3,...,\gdim$.

We observe that the matrix $\widetilde{F}^{IJ}$ define  the finite Lie-Poisson structure on $\g$.  It is well known that the  symplectic subspaces of this structure are the  orbit spaces of $\g$ under the adjoint  action and there are  $r$ global Casimirs \cite{MRbook}. Since the Slodowy slice $Q=e+\g^f$ is transversal to the orbit of $e$, the minor matrix $\widetilde{F}^{\alpha \beta}$ is nondegenerate. Let $\widetilde{F}_{\alpha \beta}$ denote its inverse.

\begin{prop} (\cite{mypaper2},\cite{BalFeh1})\label{dirac red}
The  Poisson bracket $\{.,.\}^{\widetilde{Q}}$ can be obtained by performing the Dirac reduction  on $\{.,.\}$ on ${\widetilde{Q}}$.
\end{prop}

By proposition  \ref{dirac fromula},   the leading terms of $\{.,.\}^{\widetilde{Q}}$ are given by
\begin{equation}\label{finite Poiss brac}
F^{ij}(z(x))=(\widetilde{F}^{ij}-\widetilde{F}^{i\beta} \widetilde{F}_{\beta\alpha} \widetilde{F}^{\alpha j}
)\end{equation}
\begin{equation}\label{finite Poiss brac2}
g^{ij}(z(x))= \widetilde{g}^{ij}-\widetilde{g}^{i\beta}\widetilde{ F}_{\beta\alpha}\widetilde{F}^{\alpha
j}+\widetilde{F}^{i \beta}\widetilde{F}_{\beta \alpha} \widetilde{g}^{\alpha \varphi} \widetilde{F}_{\varphi
\gamma} \widetilde{F}^{\gamma j}-\widetilde{F}^{i\beta} \widetilde{F}_{\beta \alpha} \widetilde{g}^{\alpha j}.
\end{equation}
and
\begin{equation}\label{finite Poiss brac3}
{\Gamma}^{ij}_{k}(z(x)) z_x^k=-\big(\wdt g^{i\beta} - \wdt F^{i\lambda} \wdt F_{\lambda \alpha} \wdt g^{\alpha \beta} \big)\partial_x(\wdt F_{\beta \varphi} \wdt F^{\varphi j})
\end{equation}

A consequence of this proposition is the following
\begin{prop} \cite{BalFeh1}\label{walg}The Poisson bracket $\{.,.\}^{\widetilde{Q}}$ has the following form
\begin{eqnarray}\label{leading terms}
\{z^1(x),z^1(y)\}^{\widetilde{Q}}&=& \eps \delta^{'''}(x-y) +2 z^1(x) \delta'(x-y)+ z^1_x\delta(x-y), \\\nonumber
\{z^1(x),z^i(y)\}^{\widetilde{Q}} &=& (\eta_i+1) z^i(x) \delta'(x-y)+ \eta_i z^i_x \delta(x-y),
\end{eqnarray}
$i=1,...,r+2$.
\end{prop}

Indeed, equations \eqref{leading terms} are exactly the identities which define Virasoro density and classical $W$-algebras
\cite{fehercomp}.
\subsection{The quasihomogeneity condition}
we want to study the quasihomogeneity  of the entries $g^{ij}_2(z)$.  We use the definition of the coordinates $z^I(x)$ and we assign degree $\mu_I+2$ to $\q^I(x)$ if ${\xi^I}\in \g_{\mu_I}$. These degrees agree with those given in corollary \ref{lin inv poly}. From the total order of the basis, it follows  that if $z^I(x)$ has  degree $\mu_I+2$ then degree  $\q^{\gdim-I+1}(x)$ equals  $-\mu_I+2$. Further, since $[\g_{\mu_I},\g_{\mu_J}]\subset \g_{\mu_I+\mu_J}$, an entry  $\widetilde F^{IJ}(x)$ is quasihomogenous of degree  $\mu_I+\mu_J+2$.
\begin{defn}
in considering the degrees of the coordinates $z^I$, We say a matrix $B^{IJ}(z)$ with polynomial entries is quasihomogenous of degree $n$ if each entry  $B^{IJ}(z)$ is quasihomogenous of degree $\mu_I+\mu_J+n$.
\end{defn}

\begin{prop}\cite{DamSab}\label{mat poly}
The matrix $\widetilde{F}_{\beta \alpha}(z) $  restricted to ${\widetilde{Q}}$ is polynomial and quasihomogenous of degree $-2$.
\end{prop}

\begin{prop}\label{homg of g2}
The matrix $g^{ij}(z)$ is quasihomogenous of degree $-4$ while the matrix  $F^{ij}(z)$ is quasihomogenous of degree $-2$ and  the matrix  ${\Gamma}^{ij}_{k}(z)$ is quasihomogenous of degree $ -(2\eta_k+2)-4$.
\end{prop}
\begin{proof}
The statement about the quasihomogeneity of  the matrix $F^{ij}(z)$ was proved in \cite{DamSab}. We will derive the quasihomogeneity of $g^{ij}(z)$. We know that the matrix $\widetilde g^{IJ}$ is constant antidiagonal. Hence we can write  $\widetilde g^{IJ}=C^I \delta^I_{\gdim-J+1}$ where $C^I$ are nonzero constant. In particular, we have $\widetilde g^{ij}(z)=0$. Now  consider the expression \eqref{finite Poiss brac2}. Then  for a fixed $i$ we have
\beq \label{ass eq 1}\widetilde{g}^{i\beta}\widetilde{ F}_{\beta\alpha}\widetilde{F}^{\alpha
j}= C^i \widetilde{ F}_{\gdim -i+1,\alpha}\widetilde{F}^{\alpha
j}.\eeq
Hence, the left hand sight is quasihomogenous of degree \[\mu_j+\mu_\alpha+2 - \mu_\alpha-(-\mu_i)-2=\mu_j+\mu_i=2\eta_i+2\eta_j.\]
A similar argument shows that  $ \widetilde{F}^{i\beta} \widetilde{F}_{\beta \alpha} \widetilde{g}^{\alpha j}$ is quasihomogenous of degree $2\eta_i+2\eta_j$. Finally, the summation
\beq \label{ass eq 2}\widetilde{F}^{i \beta}\widetilde{F}_{\beta \alpha} \widetilde{g}^{\alpha \varphi} \widetilde{F}_{\varphi
\gamma} \widetilde{F}^{\gamma j}=\sum_\alpha C^{\alpha} \widetilde{F}^{i \beta}\widetilde{F}_{\beta \alpha} \widetilde{F}_{ \gdim-\alpha+1,
\gamma} \widetilde{F}^{\gamma j}. \eeq
Then, it has the degree
\[\mu_i+\mu_\beta+2-\mu_\beta-\mu_\alpha-2-\mu_{\gdim-\alpha+1}-\mu_\gamma-2+\mu_\gamma+\mu_j+2=2 \eta_i+2\eta_j.\]
This proves that $g^{ij}(z)$ is quasihomogenous of degree $-4$. For the last statement in the proposition we observe that the formula for  ${\Gamma}^{ij}_{2;k}(z(x))$ is given by
\begin{equation}
{\Gamma}^{ij}_{k}(z(x)) =-\big(\wdt g^{i\beta} - \wdt F^{i\lambda} \wdt F_{\lambda \alpha} \wdt g^{\alpha \beta} \big)\partial_{z^k} (\wdt F_{\beta \varphi} \wdt F^{\varphi j}).
\end{equation}
Hence the calculation of quasihomogeneity will be same as equations \eqref{ass eq 1} and \eqref{ass eq 2} with  subtracting $2\eta_{k}+2$. This complete the proof.
\end{proof}

\subsection{Subregular classical $W$-algebra}

 Let us consider the Poisson bracket $\{.,.\}^{\widetilde{Q}}$ in  Slodowy  coordinates $(t^1,...,t^n)$ and write
  \begin{eqnarray}\label{reducedPB notations in t}
  \{t^i(x),t^j(y)\}^{[-1]} &=& F^{ij}(t(x))\delta(x-y) \\\nonumber
  \{t^i(x),t^j(y)\}^{[0]} &=& g^{ij}(t(x)) \delta' (x-y)+ \Gamma_{k}^{ij}(t(x)) t_x^k \delta (x-y)\\\nonumber
   \end{eqnarray}

\begin{prop} \label{walg1} The minor matrix  $\partial_{t^{r-1}} g^{mn}_1(t),~m,n=1,\ldots,r$ is nondegenerate and its determinant is equal to the determinant of the matrix ${1\over \rho} A_{ij}$. In particular, the minor matrix $g^{mn},~m,n=1,\ldots,r$ is generically nondegenerate. Moreover, we have the same  identities which defines classical $W$-algebras, i.e
\begin{eqnarray}\label{leading terms2}
\{t^1(x),t^1(y)\}^{\widetilde{Q}}&=& \eps \delta^{'''}(x-y) +2 t^1(x) \delta'(x-y)+ t^1_x\delta(x-y) \\\nonumber
\{t^1(x),t^i(y)\}^{\widetilde{Q}} &=& (\eta_i+1) t^i(x) \delta'(x-y)+ \eta_i t^i_x \delta(x-y).
\end{eqnarray}
\end{prop}
\begin{proof}
The nondegeneracy statements follows   from the fact that the proof of proposition \ref{nondeg}  depends only on the linear terms of the invariant differential polynomials $z^i$ (see proposition \ref{sldwy coord}). For the second part of the statement,  we need only to show that  \begin{equation}
g^{1,n}(t)= (\eta_i+1) t^i, ~ \Gamma^{1j}_{k}(t)=\eta_j \delta^j_k.
\end{equation}
 Note that, from  proposition \ref{walg}, we have
\begin{equation}
g^{1,n}(z)= (\eta_i+1) z^i, ~ \Gamma^{1j}_{ k}(z)=\eta_j \delta^j_k.
\end{equation}

If we introduce  the Euler vector field
\beq
E':=\sum_i (\eta_i+1) z^i { \partial_{z^i}}.
\eeq
Then the formula for change of coordinates   gives
\beq
g^{1j}(t)={\partial_{z^a} t^1 } {\partial_{z^b} t^j}~ g_2^{a b}(z)= E'(t^j)=(\eta_j+1) t^j.
\eeq
Where the last equality comes from quasihomogeneity of the coordinates $t^i$. For    $\Gamma^{1j}_{k}(z)$, the change of coordinates has the following formula
\beq
\Gamma^{ij}_{k}(t) d t^k=\Big({\partial_{z^a} t^i } {\partial_{z^c}\partial_{z^b} t^j} g_2^{a b}(z)+ {\partial_{z^a} t^i} {\partial_{z^b} t^j  } \Gamma^{a b}_{c}(z)\Big) d z^c.
\eeq
But then we get
\begin{eqnarray}
\Gamma^{1j}_{k} d t^k&=&\Big( E' ({\partial_{z^c} t^j  })+  {\partial_{z^b} t^j  } \Gamma^{1 b}_{c}\Big) d z^c\\\nonumber
&=& \Big( (\eta_j-\eta_c){\partial_{z^c} t^j  }+  \eta_c {\partial_{z^c} t^j  } \Big) d z^c=\eta_j {\partial_{z^c} t^j } d z^c =\eta_j d t^j
\end{eqnarray}

\end{proof}
The following theorem was proved in \cite{DamSab} using Slodowy coordinates
\begin{thm}\label{TPS}
 The  matrix $F^{ij}(t)$ is a constant multiple of the matrix
\beq\begin{pmatrix}
  0 & 0 &  \\
  0 &  \Omega \\
\end{pmatrix}
\eeq
where $\Omega$ is a $3 \times 3$ matrix of the form
\beq\begin{pmatrix}

  0& {\partial t^0 \over \partial t^{r+2}} & -{\partial t^0\over \partial t^{r+1}} \\
 -{\partial  t^0\over \partial t^{r+2}} & 0 & {\partial t^0\over \partial t^{r}} \\
 {\partial  t^0\over \partial t^{r+1}} &- {\partial  t^0\over \partial t^{r}} & 0 \\
\end{pmatrix}
\eeq
 where  $t^0$ is the restriction to $Q$ of the  invariant polynomial  $\chi^0$  defined after proposition \ref{sldwy coord}.
\end{thm}

Let $N\subset Q$ be the hypersurface  of dimension $r$ defined as follows 

\beq\label{eqs of N}
N=\Big\{t\in Q:{ \partial t^0 \over \partial t^{r+2}} = {\partial t^0\over \partial t^{r+1}}=0\Big\}
\eeq

 From the quasihomogeneity of $t^0$ and the table in page  \pageref{plystable}, we observe that ${\partial t^0 \over \partial t^{r+2}}$ depends linearly on $t^{r+2}$ and  $ {\partial t^0\over \partial t^{r+1}}$ is a polynomial in $t^{r+1}$ of degree $\iota= r-2$ (resp. $\iota=2$) if  $\g$ is a Lie algebra of type $D_r$ (resp. $E_r$). In particular,  $(t^1,....,t^r)$ is well defined   coordinates on $N$.

\begin{thm}\label{subregwalg}
The Dirac reduction of $\{.,.\}^{\widetilde{Q}}$ to ${\widetilde{N}}=\lop {N}$ is well defined and gives a local Poisson brackets $\{.,.\}^{\widetilde{N}}$. The Poisson bracket $\{.,.\}^{\widetilde{N}}$ is a classical $W$-algebra. It admits a dispersionless limit and the leading term is a nondegenerate Poisson bracket of hydrodynamic type.
\end{thm}
\begin{proof}
We observe that  the minor matrix
\beq\begin{pmatrix}
 0 & {\partial t^0\over \partial t^{r}} \\
- {\partial  t^0\over \partial t^{r}} & 0 \\
\end{pmatrix}
\eeq
of the matrix  $F^{ij}(t)$ is nondegenerate. Hence, from theorem \ref{dirac fromula}, it follows that the Dirac reduction of $\{.,.\}^{\widetilde{Q}}$ to ${\widetilde{N}}=\lop {N}$ is well defined and gives a local Poisson brackets $\{.,.\}^{\widetilde{N}}$. Let us write the reduced Poisson bracket on ${\widetilde{N}}$ in the form
\beq \{t^m(x),t^n(y)\}^{\widetilde{N}}=
\sum_{k=-1}^\infty \epsilon^k \{t^m(x),t^n(y)\}^{[k]}_{{\widetilde{N}}} \\\nonumber
\eeq
where
\begin{eqnarray}\label{reducedPB N}
   \{t^m(x),t^n(y)\}^{[-1]}_{{\widetilde{N}}} &=& \wht{F}^{mn}(t(x))\delta(x-y) \\\nonumber
  \{t^m(x),t^n(y)\}^{[0]}_{{\widetilde{N}}} &=& \wht g^{mn}(t(x)) \delta' (x-y)+ \wht\Gamma_{k}^{mn}(t(x)) t_x^k \delta (x-y).
\end{eqnarray}
Then, it follows from corollary \ref{cor dirac} that the entries  $\wht g^{mn}(t)$ and $\wht F^{ij}(t)$ equal $g^{mn}(t)$ and $F^{mn}(t)$, respectively, where  $t^{r+1}$ and $t^{r+2}$ are  solutions of equations \eqref{eqs of N}. From proposition \ref{walg1}, this implies that $\{.,.\}^{\widetilde{N}}$ is a classical $W$-algebra.  Moreover, from proposition \ref{TPS}, we have  $\wht F^{ij}=0$. Hence $\{.,.\}^{\widetilde{N}}$ admits a dispersionless  limit. Furthermore,  proposition \ref{walg1} implies that  $\wht g^{mn}(t(x))$ is generically nondegenerate. Hence, $\{t^m(x),t^n(y)\}^{[0]}_{{\widetilde{N}}}$ is  nondegenerate Poisson brackets of hydrodynamics type.
\end{proof}

  In addition to the fact that $\{.,.\}^{\widetilde{N}}$ gives a Frobenius structure. The construction  by considering  the theory of opposite Cartan subalgebra implies that it is very associated to the Drinfeld-Sokolov hierarchy obtained  in \cite{gDSh1} that $\{.,.\}^{\widetilde{Q}}$. Therefore, we call it \textbf{subregular classical $W$-algebra}.
  
\section{Algebraic Frobenius manifold}
In this section we obtain the promised algebraic Frobenius structure.

Let us consider, using the Dubrovin-Novikov theorem \ref{DN thm}, the  contravariant  metric $\wht g^{mn}(t)$ on $N$ and its Levi-Civita connection $\wht\Gamma_{k}^{mn}(t)$. From  proposition \ref{homg of g2}, these matrices  are linear in $t^{r-1}$. Hence,  lemma \ref{flat pencil} implies that  the matrices
 \beq
\wht g^{mn}_{2}(t)= \wht g^{mn}(t),~~~~ \wht g^{mn}_{1}(t)=\partial_{t^{r-1}}\wht g^{mn}_{2}(t)
 \eeq
form  a flat pencil of metrics on $N$.

\begin{prop}
There exist    quasihomogenous polynomials coordinates of degrees $2 \eta_i+2 $ in the form
\[ s^i=t^i+T^i(t^1,...,t^{i-1})\] such that  the matrix $\wht g_1^{ij}(s)$   is  constant antidiagonal. Furthermore, in this coordinates the metric $g^{ij}_2(s)$ and its Levi-Civita connection have the following entries
\begin{equation}
g^{1,n}_2(s)= (\eta_i+1) s^i, ~ \Gamma^{1j}_{2 k}(s)=\eta_j \delta^j_k
\end{equation}
\end{prop}
\begin{proof}
The proof of the first part of the proposition is given in \cite{DCG} using the quasihomogeneity property of the matrix $\wht g^{mn}$. The second part is obtained in the same manor as in proposition \ref{walg}.
\end{proof}

We  assume without lost of generality that the  coordinates $t^i$ are the flat coordinates for $\wht g^{ij}_1$.
\begin{thm}\label{my thm}
The flat pencil of metrics given by $ \wht g^{mn}_{1}(t)$ and $ \wht g^{mn}_{2}(t)$  on the space  $N$ is regular quasihomogenous of degree $d={\kappa-1\over \kappa+1}$.
\end{thm}
\begin{proof}
In the notations of definition \ref{def reg} we take $\tau ={1\over \kappa+1}t^1$ then
\begin{eqnarray}
 E&=& g^{ij}_2 {\partial_{t^j} \tau  }~{\partial_{t^i}  }={1\over \kappa+1} \sum_{i} (\eta_i+1) t^i{\partial_{t^i} },\\\nonumber
 e &=&  g^{ij}_1 {\partial_{t^j} \tau }~{\partial_{t^i}  }={\partial_{t^{r-1}} }.
  \end{eqnarray}
We see immediately  that   \[ [e,E]=e\]
The identity \begin{equation} \Lie_e (~,~)_2 =
(~,~)_1 \end{equation} follows from  the fact that ${\partial_{t^{r-1}} }={\partial_{ z^{r-1}}}$. Then
\begin{equation}
\Lie_e(~,~)_1
=0.
\end{equation}
 is a consequence  from the quasihomogeneity of the matrix $g_1^{ij}$ (see lemma \ref{homg of g1}). We also obtain from proposition \ref{homg of g2} that
  \begin{equation}
\Lie_E(~,~)_2(dt^i,dt^j)= E(g^{ij}_2)-{\eta_i+1\over \kappa+1}g^{ij}_2-{\eta_j+1\over \kappa+1} g^{ij}_2={-2\over \kappa+1} g^{ij}_2.
\end{equation}
Hence,
\begin{equation} \Lie_E (~,~)_2 =(d-1) (~,~)_2
\end{equation}
It remains to prove the regularity condition. But the (1,1)-tensor is nondegenerate since it has the entries
\begin{equation}
  R_i^j = {d-1\over 2}\delta_i^j + {\nabla_1}_iE^j = {\eta_i \over \kappa+1} \delta_i^j.
\end{equation}
This complete the proof.
\end{proof}

Now we have all the tools to prove the following
\begin{thm}\label{main thm}
  The space $N$ has a natural  structure of algebraic Frobenius manifold   with charge  $\kappa-1\over \kappa+1$ and degrees $\eta_i+1\over \kappa+1$,~ $i=1,...,r$.
\end{thm}
\begin{proof}
It follows from  theorem \ref{my thm} and \ref{dub flat pencil}  that $N$  has a Frobenius structure of degree $\kappa-1\over \kappa+1$. This Frobenius structure is algebraic  since in the coordinates  $t^i$ the potential $\mathbb{F}$ is constructed using   equations \eqref{frob eqs}. Besides we have from theorem \ref{subregwalg}  the matrix $\wht g^{mn}_2$  depends on the nontrivial solutions of equations \eqref{eqs of N}.
\end{proof}

\subsection{The algebraic Frobenius manifold of  $D_4(a_1)$}\label{D4}

  We verify the procedure, outlined in this work, of constructing algebraic Frobenius manifold when $\g$ is of type $D_4$.  For this end we choose  the  realization of $D_4$ as a subalgebra of  $gl_8(\mathbb{C})$ given in the appendix of \cite{DS}.  In this case it is easy to obtain a representation of $e$ which belongs to strictly lower diagonal matrices. In what follows we will denote by $\sigma_{i,j}$ the standard matrix defined by $(\sigma_{i,j})_{k,l}=\delta_{i,k} \delta_{j,l}\in gl_8(\mathbb{C})$.  We fix the subregular nilpotent element $e$ and the $sl_2$-triple $\{e,h,f\}$  as follows
\begin{eqnarray}
e &=&\sigma _{2,1}-\sigma _{3,1}+\sigma _{4,3}-\frac{\sigma _{5,2}}{2}+\sigma
   _{6,5}+\frac{\sigma _{7,4}}{2}+\sigma _{8,6}+\sigma _{8,7}
\end{eqnarray}

\begin{eqnarray}
h &=& -4 \sigma _{1,1}-2 \sigma _{2,2}-2 \sigma _{3,3}+2 \sigma _{6,6}+2 \sigma _{7,7}+4
   \sigma _{8,8}
\end{eqnarray}
\begin{eqnarray}
f &=&2 \sigma _{1,2}-2 \sigma _{1,3}-2 \sigma _{2,4}-8 \sigma _{2,5}+4 \sigma _{3,4}+4
   \sigma _{3,5}+4 \sigma _{4,6}\\\nonumber & & +8 \sigma _{4,7}+4 \sigma _{5,6}+2 \sigma _{5,7}+2
   \sigma _{6,8}+2 \sigma _{7,8}
\end{eqnarray}
We observe that $Wt(e)=\{1,3,3,1,1,3\}$ and $Et(e)=\{1,3,3,1\}$. We construct a  basis for $\g$  satisfy the hypotheses  of proposition \ref{reg:sl2:normalbasis} from the formula
\beq
X_I^i= {1\over (\eta_i+I)}{\rm ad}^{\eta_i+I} e~ X_{-\eta_i}^i,~~~ i=1,...,6.
\eeq
where the lowest root vectors   $X_{-\eta_i}^i$ are
\begin{eqnarray}
X^2_{-3} &=& 24 \sqrt{3} \sigma _{3,8}-24 \sqrt{3} \sigma _{1,6}\\\nonumber
X^3_{-3} &=& -24 \sigma_{1,6}-48 \sigma _{1,7}-48 \sigma _{2,8}+24 \sigma _{3,8}
\\\nonumber
X^4_{-1} &=&-4 \sqrt{\frac{3}{5}} \sigma _{1,2}-2 \sqrt{\frac{3}{5}} \sigma _{1,3}+2
   \sqrt{\frac{3}{5}} \sigma _{2,4}+2 \sqrt{\frac{3}{5}} \sigma _{3,4}-12
   \sqrt{\frac{3}{5}} \sigma _{3,5}\\\nonumber & &-12 \sqrt{\frac{3}{5}} \sigma _{4,6}+2
   \sqrt{\frac{3}{5}} \sigma _{5,6}-2 \sqrt{\frac{3}{5}} \sigma _{5,7}+2
   \sqrt{\frac{3}{5}} \sigma _{6,8}-4 \sqrt{\frac{3}{5}} \sigma _{7,8}
\\\nonumber
X^5_{-1} &=& -\frac{8 \sigma _{1,2}}{\sqrt{5}}-2 \sqrt{5} \sigma _{1,3}-2 \sqrt{5} \sigma
   _{2,4}+\frac{8 \sigma _{2,5}}{\sqrt{5}}+\frac{2 \sigma _{3,4}}{\sqrt{5}}-\frac{4
   \sigma _{3,5}}{\sqrt{5}}\\\nonumber& &-\frac{4 \sigma _{4,6}}{\sqrt{5}}-\frac{8 \sigma
   _{4,7}}{\sqrt{5}}+\frac{2 \sigma _{5,6}}{\sqrt{5}}+2 \sqrt{5} \sigma _{5,7}+2
   \sqrt{5} \sigma _{6,8}-\frac{8 \sigma _{7,8}}{\sqrt{5}}.
\\\nonumber
X^6_{-2} &=&-4 \sqrt{3} \sigma _{1,4}+8 \sqrt{3} \sigma _{1,5}+8 \sqrt{3} \sigma _{2,6}+8
   \sqrt{3} \sigma _{3,7}+8 \sqrt{3} \sigma _{4,8}-4 \sqrt{3} \sigma _{5,8} .
\end{eqnarray}
The opposite Cartan subalgebra $\h'$ have the following normalized basis
\begin{eqnarray}
y_1 &=&e+ X_{-3}^4 \\\nonumber
y_2 &=& -X^2_3-{3\over \sqrt{5}}X^4_{-1}-{1\over 5 \sqrt{3}} X^3_{-1}+{1\over 2} X^6_{-1}\\\nonumber
y_3 &=& -X^3_3+3 X_{-1}^1-{1\over 5 \sqrt{3}}X^2_{-1}+{3\over \sqrt{5}} X_{-1}^5\\\nonumber
y_4 &=&-X^4_1-{1\over \sqrt{5}} X_{-3}^2.
\end{eqnarray}
The matrix of the  restriction of  $\bil . .$ to $\h'$ under the order $\{y_1,y_4,y_2,y_3\}$ equals
\begin{equation}
\left(
\begin{array}{llll}
 0 & 0 & 0 & 4 \\
 0 & 0 & -\frac{4}{\sqrt{5}} & 0 \\
 0 & -\frac{4}{\sqrt{5}} & 0 & 0 \\
 4 & 0 & 0 & 0
\end{array}
\right)
\end{equation}
We write an element $z$ in Slodowy slice $ Q$ in the form
\beq
z=z_1 X_{-1}^1+ z_2 X_{-3}^3+ z_3 X_{-3}^4 +z_4 X_{-1}^2+ z_5 X_{-1}^5+ z_6 X_{-2}^6+ e.
\eeq
Here we lower the index for convenience.  Then the  restriction of the invariant polynomials to $Q$ can be found form the coefficients of the indeterminant $P$ in the equation $\det (z-P)$. After normalization we get the following Slodowy coordinates on $Q$
\begin{eqnarray}
t_1 & =& z_1\\\nonumber
t_2 & =& z_2-\frac{1}{2} \sqrt{3} z_1^2+\frac{4 z_5 z_1}{\sqrt{15}}+\frac{7 z_4^2}{5 \sqrt{3}}-\frac{7 z_5^2}{5 \sqrt{3}}\\\nonumber
t_3&=& z_3-\frac{3 z_1^2}{2}-\frac{4 z_4 z_1}{\sqrt{15}}-\frac{14 z_4 z_5}{5 \sqrt{3}}\\\nonumber
t_i&=&z_i , ~i=4,5,6.
\end{eqnarray}
We take the following as the  restriction to $Q$ of a highest degree invariant polynomial
\beq
\begin{split}
t_0=&20 t_1^3+18 \sqrt{15} t_4 t_1^2-18 \sqrt{5} t^5 t_1^2+60 t_4^2 t_1+60 t_5^2 t_1+6 \sqrt{3} t_2 t_1\\&+18 t_3 t_1-20 \sqrt{5} t_5^3-27 t_6^2+12 \sqrt{15} t_3 t_4+60
   \sqrt{5} t_4^2 t_5-12 \sqrt{15} t_2 t_5
\end{split}
\eeq
Note that in the case $t_i=0,~i=1,2,3$ we get the equation
\beq
f(t_4,t_5,t_6)=-20 \sqrt{5} t_5^3+60 \sqrt{5} t_4^2 t_5-27 t_6^2
\eeq
which define a simple hypersurface singularity of type $D_4$.

It follows that the leading term of the classical $W$-algebra on $\widetilde{Q}$ is given by
 \beq F^{ij}(t)=75 \begin{pmatrix}
  0 & 0 &  \\
  0 &  \Omega \\
\end{pmatrix}
\eeq
where $\Omega$ is a $3 \times 3$ matrix of theorem \ref{TPS}.
The hypersurface $N\subset Q$ is defined by the equations
\begin{eqnarray} {\partial t_0\over \partial t_{6}}&=&270t_6=0\\\nonumber
                 {\partial t_0\over \partial t_{5}}&=&-5 \left(-18 \sqrt{5} t_1^2+120 t_5 t_1+60 \sqrt{5} t_4^2-60 \sqrt{5} t_5^2-12 \sqrt{15} t_2\right)=0
\end{eqnarray}
We choose the  flat coordinates
\begin{eqnarray}
s_1&=&t_1\\\nonumber
s_2&=&t_2+ \frac{\sqrt{3} }{4} t_1^2-\frac{5\sqrt{3}{4}} t_4^2\\\nonumber
s_3&=&t_3+ \frac{3}{2}t_1^2+\frac{ \sqrt{15}}{2} t_4 t_1\\\nonumber
s_4&=&t_4
\end{eqnarray}
Then the potential $\mathbb{F}$ of the Frobenius structure  reads
\begin{eqnarray}
\mathbb{F}&=&\frac{Z}{180} \left(-\sqrt{5} s_1^4-10 \sqrt{5} s_4^2 s_1^2+8 \sqrt{15} s_2 s_1^2-25 \sqrt{5} s_4^4-48
   \sqrt{5} s_2^2+40 \sqrt{15} s_2 s_4^2\right)\\\nonumber
   &+&\frac{1}{2880} \Big(35 s_1^5+510 s_4^2 s_1^3-48 \sqrt{3} s_2 s_1^3+775 s_4^4 s_1+360 s_3^2 s_1\\\nonumber & &-720 \sqrt{5} s_2 s_3 s_4+~1128 s_2^2 s_1-1840 \sqrt{3}
   s_2 s_4^2 s_1\Big)
\end{eqnarray}
where $Z$ is a solution of the quadratic equation
\begin{equation}
Z^2-\frac{2}{\sqrt{5}}s_1 Z+\frac{3}{20}s_1^2-\frac{1}{4}s_4^2+\frac{\sqrt{3}}{5} s_2=0.
\end{equation}
It is straightforward to check validity of the WDVV equations for this potential. The identity vector field is ${\partial\over{\partial s_3}}$ and the quasihomogeneity reads
\begin{equation}
{1\over 2} s_1 {\partial \mathbb{F}\over \partial {s_1}}+  s_2 {\partial \mathbb{F}\over \partial {s_2}}+ s_3 {\partial \mathbb{F}\over \partial {s_3}}+{1\over 2} s_4 {\partial \mathbb{F}\over \partial {s_4}}={5 \over 2} \mathbb{F}.
\end{equation}

\section{Conclusions and remarks}

In this work we obtained infinite number of examples of algebraic Frobenius manifolds. These examples  correspond to the regular quasi-Coxeter conjugacy classes $D_r(a_1)$ where $r$ is even and $E_r(a_1)$. One of the tools we use is the structure of opposite Cartan subalgebra which relate the subregular nilpotent orbit to the conjugacy class. The structure of opposite Cartan subalgebra exists only for  regular conjugacy class. But taking the subregular nilpotent orbit in the Lie algebra $D_5$ we obtain algebraic Frobenius manifold related to the nonregular  quasi-Coxeter conjugacy class $D_5(a_1)$. This implies that the existence of algebraic Frobenius manifold is a far deeper than the notion of opposite Cartan subalgebra. This fact  will be the frame work of our future research. Our next step is to develop  a method  to uniform the construction of all algebraic Frobenius manifolds that could be obtained from  quasi-Coxeter conjugacy classes in Weyl groups.

In this work we give, for the first time, a geometric realization of algebraic Frobenius manifolds. The examples obtained are  certain  hypersurfaces in the total spaces of semi-universal deformations of simple hypersurface singularities. We hope this will rich the relation between Frobenius manifolds and singularity theory. Which one of its main contributions is the existence of  polynomial Frobenius structures   on the  universal unfolding of simple hypersurface singularities. For the definition and  deference between semiuniversal and unfolding see chapter 2 section 1 of \cite{Gruel}.

\noindent{\bf Acknowledgments.}

The author thanks  B. Dubrovin for useful discussions. This work was done primarily during the author  postdoctoral fellowship at the Abdus Salam International Centre for Theoretical Physics (ICTP). Italy.

\end{document}